\documentclass[12pt]{amsart}
\usepackage{amssymb,amscd}
\usepackage{verbatim}

\let\frak\mathfrak
\let\Bbb\mathbb

\def\>{\relax\ifmmode\mskip.666667\thinmuskip\relax\else\kern.111111em\fi}
\def\<{\relax\ifmmode\mskip-.333333\thinmuskip\relax\else\kern-.0555556em\fi}
\def\vsk#1>{\vskip#1\baselineskip}
\def\vv#1>{\vadjust{\vsk#1>}\ignorespaces}
\def\vvn#1>{\vadjust{\nobreak\vsk#1>\nobreak}\ignorespaces}

  \let\ssize\scriptstyle
\let\sssize\scriptscriptstyle

\let\Medskip\medskip
\def\medskip{\par\Medskip}
\let\Bigskip\bigskip
\def\bigskip{\par\Bigskip}

\let\Maketitle\maketitle
\def\maketitle{\Maketitle\thispagestyle{empty}\let\maketitle\empty}

\newtheorem{thm}{Theorem}[section]
\newtheorem{cor}[thm]{Corollary}
\newtheorem{lem}[thm]{Lemma}

\numberwithin{equation}{section}

\theoremstyle{definition}
\newtheorem*{rem}{Remark}

\let\mc\mathcal
\let\nc\newcommand

\let\al\alpha
\let\bt\beta

\let\la\lambda

\let\phi\varphi

\let\Si\Sigma

\let\om\omega

\let\der\partial

\let\geq\geqslant

\let\on\operatorname
\let\bi\bibitem
\let\bs\boldsymbol

\def\C{{\mathbb C}}
\def\Z{{\mathbb Z}}

\def\F{{\mc F}}

\def\+#1{^{\{#1\}}}

\def\End{\on{End}}

\def\Res{{\on{Res}}}

\def\gln{\mathfrak{gl}_N}
\def\sln{\mathfrak{sl}_N}

\def\beq{\begin{equation}}
\def\eeq{\end{equation}}
\def\be{\begin{equation*}}
\def\ee{\end{equation*}}

\nc{\bea}{\begin{eqnarray*}}
\nc{\eea}{\end{eqnarray*}}
\nc{\bean}{\begin{eqnarray}}
\nc{\eean}{\end{eqnarray}}
\nc{\Ref}[1]{{\rm(\ref{#1})}}

\let\ga\gamma

\nc{\Il}{{\mc I_{\bs\la}}}
\nc{\bla}{{\bs\la}}
\nc{\Fla}{\F_\bla}
\nc{\tfl}{{T^*\Fla}}
\nc{\GL}{{GL_n(\C)}}
\nc{\GLC}{{GL_n(\C)\times\C^*}}

\let\sd s 

\def\ddk_#1{\kk_{#1}\<\>\frac\der{\der\<\>\kk_{#1}}}

\def\bul{\mathbin{\raise.2ex\hbox{$\sssize\bullet$}}}
\def\intt{\mathchoice
{\mathop{\raise.2ex\rlap{$\,\,\ssize\backslash$}{\intop}}\nolimits}
{\mathop{\raise.3ex\rlap{$\,\sssize\backslash$}{\intop}}\nolimits}
{\mathop{\raise.1ex\rlap{$\sssize\>\backslash$}{\intop}}\nolimits}
{\mathop{\rlap{$\sssize\<\>\backslash$}{\intop}}\nolimits}}

\let\kk q 
\let\cc c

\let\Ko K

\def\GZ/{Gelfand-Zetlin}
\def\KZ/{{\slshape KZ\/}}
\def\qKZ/{{\slshape qKZ\/}}
\def\XXX/{{\slshape XXX\/}}

\nc{\slnl}{{\sln (\lambda)}}
\nc{\PCN}{{   (\C[x])^N   }}
\nc{\di}{\on{Diag}}
\nc{\dio}{\on{Diag}_0}
\nc{\Mm}{{\mc M}}
\nc{\Nn}{{\mc N}}

\nc{\A}{{\mc C}}

\nc{\PCr}{{  P  (\C[x])^n   }}

\nc{\Pk}{{(\bs{P}^1)^k}}

\nc{\N}{{\Bbb N}}

\nc{\Ll}{{\mc L}}

\nc{\ord}{{\on{ord}\,}}

\newcommand{\OS}{\mathcal {A}}
\def\FF{{\mathcal F}}

\nc{\Sing}{{\on{Sing}\,}}
\nc{\sing}{{\on{Sing}\,}}

\nc{\Hess}{{\on{Hess}}}

\nc{\R}{{\Bbb R}}
\newcommand{\Pee}{{\mathbb P}}
\let\on\operatorname
\nc{\Kk}{{\bs K}}
\nc{\Ap}{{A_\Phi(z)}}
\nc{\ap}{{A_\Phi(z)}}

\nc{\sv}{{\on{Sing}_a V}}
\nc{\cd}{{\C^n-\Delta}}
\nc{\UT}{{U^0}}   
\nc{\Spect}{\on{Spec}\nolimits}

\nc\z{{\bs z}}
\nc\p{{\bs p}}
\nc\q{{\bs q}}
\nc\ttt{{\bs t}}
\nc\OC{{\mc O(C_{\A,a})}}
\nc\Ia{{\mc I}}
\nc\OCx {{\mc O(C_{\A(x),a})}}
\nc\Cs{{(\C^n)^*}}
\nc\OLx{{ \mc O(L_{Y,a}(x))}}

\begin{document}

\hrule width0pt
\vsk->

\title[Critical set of the master function ]
{Critical set of the master function and characteristic variety of the associated Gauss-Manin differential equations
}

\author
[ A.\,Varchenko]
{ A.\,Varchenko$\>^\diamond$}

\maketitle

\begin{center}

\vsk.5>
{\it Department of Mathematics, University
of North Carolina at Chapel Hill\\ Chapel Hill, NC 27599-3250, USA\/}

\end{center}

{\let\thefootnote\relax
\footnotetext{\vsk-.8>\noindent
$^\diamond\<${\sl E\>-mail}: \enspace anv@email.unc.edu\>,
supported in part by NSF grant DMS-1362924}}

\medskip

\begin{abstract}

We consider a weighted family of $n$ parallelly transported hyperplanes in a $k$-dimensional
 affine space  and describe  the characteristic variety of the Gauss-Manin
differential equations for associated hypergeometric integrals.
The characteristic variety is given as the zero set of Laurent polynomials, whose coefficients are determined by weights and
 the associated point in the Grassmannian Gr$(k,n)$.  The Laurent polynomials are in involution.
These statements generalize  \cite{V7}, where
such a description was obtained for a weighted {\it generic} family of  parallelly transported hyperplanes.

An intermediate object between the differential equations and the characteristic variety is
the algebra of functions on the critical set of the associated master function.
We construct a linear isomorphism between the vector space of the Gauss-Manin
differential equations and the algebra of functions.
The isomorphism allows us to describe the characteristic variety.
It also allowed us to  define an integral structure on the vector space of the algebra
and  the associated (combinatorial) connection on the family of such algebras.

\end{abstract}

{\small \tableofcontents  }

\setcounter{footnote}{0}
\renewcommand{\thefootnote}{\arabic{footnote}}

\section{Introduction}

There are three places, where a flat connection depending on a parameter appears:

\noindent
$\bullet$\  KZ equations,
$\kappa \frac{\der I}{\der z_i} = K_i I$,   $i=1,\dots,n$.
Here $\kappa$ is a parameter, $I(z_1,\dots,z_n)$  a $V$-valued function, where  $V$ is a vector space from representation theory,
$K_i : V\to V$ are linear operators, depending on $z$. The connection is flat for all $\kappa$, see, for example, \cite{EFK, V2}.

\noindent
$\bullet$\ Quantum differential equations,
$\kappa \frac{\der I}{\der z_i} = p_i *_z I$,   $i=1,\dots,n.$
Here $p_1,\dots,p_n$ are generators of some commutative algebra $H$ with quantum multiplication $*_z$ depending on $z$.
The connection is flat for all $\kappa$.
These equations are a part of the Frobenius structure on the quantum cohomology of a variety, see \cite{D, M}.

\noindent
$\bullet$\ Differential equations for hypergeometric integrals associated with a family of weighted arrangements with parallelly
 transported hyperplanes,
$\kappa \frac{\der I}{\der z_i} = K_i I$,   $ i=1,\dots,n$. The connection is flat for all $\kappa$,
see, for example, \cite{V1, OT2}.

\smallskip
If $\kappa \frac{\der I}{\der z_i} = K_i I$, $i=1,\dots,n$, is a system of $V$-valued differential equations of one of these types, then
its characteristic variety is
\be
\Spect = \{(q,p)\in T^*\C^n\ |\ \exists
v\in V-\{0\}\ \text{with}\ K_j(q)v = p_jv,\ j=1,\dots,n\}.
\ee
It is known that the characteristic varieties of the first two types of differential equation are interesting objects. For example,
the characteristic variety of the quantum differential equation of the flag variety is the zero set of the Hamiltonians of the classical Toda lattice,
see \cite{G, GK}, and the characteristic variety of the $\gln$  KZ equations with values in the tensor power of the vector representation  is the zero
set of the Hamiltonians of the classical Calogero-Moser system, see \cite{MTV2}.

In this paper we describe the characteristic variety of the Gauss-Manin differential equations for hypergeometric integrals associated with
an {\it arbitrary}  weighted family of $n$ parallelly transported hyperplanes in $\C^k$. This description generalizes \cite[Corollary 4.2]{V7}, where
such a description was obtained for a weighted {\it generic} family of  parallelly transported hyperplanes.

The characteristic variety is given as the zero set of Laurent polynomials, whose coefficients are determined by weights and
 the associated point in the Grassmannian Gr$(k,n)$.
The Laurent polynomials are in involution, see Section \ref{sec def of Lag}.

It is known that the KZ differential equations, as well as some  quantum differential equations,
 can be identified with certain symmetric parts of the Gauss-Manin differential equations of weighted families of
parallelly transported hyperplanes, see \cite{SV, TV}.
Therefore, the results of this paper on the characteristic variety  is  a  step to studying characteristic varieties of more general KZ and quantum
 differential equations, which admit integral hypergeometric representations.

Our description of the characteristic variety is based on the fact \cite{V4},
that the characteristic variety of the Gauss-Manin differential equations is generated by the master function
of the corresponding hypergeometric integrals, that is, the characteristic variety coincides with the Lagrangian variety
of the master function. That fact was  developed later in \cite[Theorem 5.5]{MTV1}, it was proved there with the help of the Bethe ansatz, that
the local algebra of a critical point of the master function associated with a $\gln$ KZ equation can be identified with
a suitable local Bethe algebra of the corresponding $\gln$ module.

\vsk.3>
In Section \ref{Sec Arrangements}, we consider a weighted arrangement $(\A,a)$ of $n$ affine hyperplanes in $\C^k$. Here $a$ is a point of $(\C^\times)^n$
called the weight.    We introduce the Aomoto complex
$(\OS(\A),d^{(a)})$, the flag complex $(\FF(\A),d)$, the critical set $C_{\A,a}$ of the master function on the complement to the arrangement.
 We remind the isomorphism of vector spaces $\mc E : \OC \to H^k(\FF(\A),d)$,  constructed
in \cite{V5} and given by a variant of  the Grothendieck residue. The algebra $\OC$ is a nonlinear object, defined by the critical point equations;
the space $H^k(\FF(\A),d)$ is a combinatorial object defined by the matroid of the arrangement  $\A$; the isomorphism $\mc E $ is given by a $k$-dimensional integral.
Our first main result, Theorem \ref{1st main thm},  gives an  elementary isomorphism $[\mc S^{(a)}] :  H^k(\FF(\A),d) \to \OC$
such that $[\mc S^{(a)}]\circ \mc E = (-1)^k$.  Theorem \ref{1st main thm} allows us to bring to $\OC$ the combinatorial structures on $H^k(\FF(\A),d)$.
In particular, we construct a set $\{w_{\al_0,\dots,\al_k}\}$ of {\it marked} elements  of $\OC$, labeled by flags of edges  of $\A$,
which  spans  the vector space  $\OC$. We give the linear relations between the marked elements; the relations are with integer coefficients
and depend only on the matroid of $\A$, see
\linebreak
Corollary \ref{marked span}.

In Section \ref{sec par trans},  we consider a family of weighted arrangements $(\A(x),a)$  of $n$ affine hyperplanes in $\C^k$, parameterized by
$x\in\C^n$. The hyperplanes of $(\A(x),a)$  are transported parallelly as $x$ changes. Each of the arrangements
has  the algebra $\OCx$  of functions on the critical set $C_{\A(x),a}$ of the associated master function.
We define the discriminant $\Delta\subset\C^n$ so that the combinatorics of $\A(x)$ does not change when $x$ runs through $\cd$.
The constructions of Section \ref{Sec Arrangements} provide us with the vector bundle of algebras
$\sqcup_{x\in\C^n-\Delta} \OCx \to \C^n-\Delta$, whose fibers are canonically identified with the help of the marked elements.
The multiplication in $\OCx$ depends on $x$. The isomorphism $H^k(\FF(\A(x)),d) \to \OCx$ of Section \ref{Sec Arrangements}
 allows us to describe the multiplication in $\OCx$ combinatorially, see Corollary \ref{cor *_x K_j}.
We describe the Gauss-Manin differential equations associated with the weighted family of arrangements as a system of differential equations
on the bundle of algebras.

In Section  \ref{sec def of Lag}, for given $k$-dimensional
vector subspace $Y\subset \C^n$ and weight $a\in (\C^\times)^n$,
we define a Lagrangian variety $L_{Y,a}\subset \C^n\times \Cs$.  We consider the projection
$\pi_{L_{Y,a}}:L_{Y,a}\to\C^n$ and the algebras of functions $\OLx$ on fibers of the projection.
We describe $L_{Y,a}$ as the zero set of Laurent polynomials in involution.

In  Section \ref{An arrangement second part},
for the family of arrangements $\A(x)$ considered in Section \ref{sec par trans},
we define a $k$-dimensional subspace $Y\subset \C^n$  and construct an isomorphism
$\Psi_{\A(x),a}^*\  :\  \OLx \to \OCx$ of algebras for any $x\in \C^n$. Theorem  \ref{thm 1}, on this isomorphism,
is our second main result. We discuss corollaries of Theorem \ref{thm 1} in Sections \ref{H J}-\ref{Real solutions}.
In particular, in Corollary \ref{cor hess jac} we describe the ratio of the Hessian element in $\OCx$ and the Jacobian element in
$\OLx$ and in Corollary \ref{res res} we identify the standard residue form on $\OCx$, defined by a $k$-dimensional integral,
 with a residue form on  $\OLx$, defined by an $n$-dimensional integral.
In Section \ref{cor 4.5}, we consider the vector bundle of algebras
$\sqcup_{x\in\C^n-\Delta} \OLx \to \C^n-\Delta$. The isomorphism $\Psi_{\A(x),a}^*$ allows us to identify  the fibers of the bundle
and  describe the Gauss-Manin differential equations with values in that bundle. They have the form
$\kappa \frac{\der I}{\der q_j}(x) = [p_j] *_x I$, $j=1,\dots,n$, where $q_1,\dots,q_n$ are coordinates on $\C^n$, $p_1,\dots,p_n$ are the dual coordinates
on $\Cs$,  $[p_j] *_x$ is the multiplication by $p_j$ in $\OLx$.

In Section \ref{Real solutions}, we observe a rather unexpected 'reality' property  of the Lagrangian variety $L_{Y,a}$, which is similar to the reality
property of Schubert calculus,  see \cite{MTV2, MTV3, So}.

In Theorem \ref{char var thm}, we identify the characteristic variety of the Gauss-Manin differential equations associated with the family of arrangements
considered in Section \ref{sec par trans} and the Lagrangian variety $L_{Y,a}$ defined in  Section \ref{An arrangement second part}.


\medskip
The author thanks V.\,Tarasov for collaboration. The proof of Theorem \ref{1st main thm} is the result of joint efforts.
The author thanks B.\,Dubrovin and A.\,Veselov for helpful discussions.

\section{Arrangements}
\label{Sec Arrangements}

\subsection{Affine arrangement}
\label{An affine arrangement}

Let $k,n$ be positive integers, $k<n$. Denote $J=\{1,\dots,n\}$.

Consider the complex affine space $\C^k$ with coordinates $t_1,\dots,t_k$.
Let $\A =(H_j)_{j\in J}$,  be an arrangement of $n$ affine hyperplanes in
$\C^k$. Denote $U(\A) = \C^k - \cup_{j\in J} H_j$,
the complement.
An {\it edge} $X_\al \subset \C^k$ of $\A$ is a nonempty intersection of some
hyperplanes  of $\A$. Denote by
 $J_\al \subset J$ the subset of indices of all hyperplanes containing $X_\al$.
Denote $l_\al = \mathrm{codim}_{\C^k} X_\al$.

We assume that $\A$ is {\it essential}, that is, $\A$ has a vertex, an edge which is a point.

An edge is called {\it dense} if the subarrangement of all hyperplanes containing
it is irreducible: the hyperplanes cannot be partitioned into nonempty
sets so that, after a change of coordinates, hyperplanes in different
sets are in different coordinates. In particular, each hyperplane of
$\A$ is a dense edge.

\subsection{Orlik-Solomon algebra}
Define complex vector spaces $\OS^p(\A)$, $p = 0,  \dots, k$.
 For $p=0$, we set $\OS^p(\A)=\C$. For  $p \geq 1$,\
 $\OS^p(\A)$   is generated by symbols
$(H_{j_1},...,H_{j_p})$ with ${j_i}\in J$, such that
\begin{enumerate}
\item[(i)] $(H_{j_1},...,H_{j_p})=0$
if $H_{j_1}$,...,$H_{j_p}$ are not in general position, that is, if the
intersection $H_{j_1}\cap ... \cap H_{j_p}$ is empty or
 has codimension
 less than $p$;
\item[(ii)]
$ (H_{j_{\sigma(1)}},...,H_{j_{\sigma(p)}})=(-1)^{|\sigma|}
(H_{j_1},...,H_{j_p})
$
for any element $\sigma$ of the
symmetric group $ \Sigma_p$;
\item[(iii)]
$\sum_{i=1}^{p+1}(-1)^i (H_{j_1},...,\widehat{H}_{j_i},...,H_{j_{p+1}}) = 0
$
for any $(p+1)$-tuple $H_{j_1},...,H_{j_{p+1}}$ of hyperplanes
in $\A$ which are
not in general position and such that $H_{j_1}\cap...\cap H_{j_{p+1}}\not = \emptyset$.
\end{enumerate}
The direct sum $\OS(\A) = \oplus_{p=1}^{N}\OS^p(\A)$ is the {\it Orlik-Solomon}
 algebra with respect to  multiplication
 $ (H_{j_1},...,H_{j_p})\cdot(H_{j_{p+1}},...,H_{j_{p+q}}) =
 (H_{j_1},...,H_{j_p},H_{j_{p+1}},...,H_{j_{p+q}})$.

\subsection{Aomoto complex }
\label{sec weights}
 Fix a point $a=(a_1,\dots,a_n)\in (\C^\times)^n$ called the {\it weight}.
Then the arrangement $\A$ is {\it weighted}:\  for $j\in J$, we assign weight $a_j$ to hyperplane $H_j$.
For an edge $X_\al$, define its weight
$a_\al = \sum_{j\in J_\al}a_j$.  We define  $\om^{(a)} = \sum_{j\in J} a_j\cdot (H_j)\ \in  \OS^1(\A)$.
 Multiplication by $\om^{(a)}$ defines the differential
 $ d^{(a)} :   \OS^p(\A) \to \OS^{p+1}(\A)$,  $x  \mapsto \om^{(a)}\cdot x$,
   on  $\OS(\A)$, \ $(d^{(a)})^2=0$. The complex $(\OS(\A), d^{(a)})$ is called the {\it Aomoto complex}.

\subsection{Flag complex, see \cite{SV} }

 For an edge $X_\alpha$, \ $l_\alpha=p$, a {\it flag} starting at $X_\alpha$ is a sequence
$ X_{\alpha_0}\supset
X_{\alpha_1} \supset \dots \supset X_{\alpha_p} = X_\alpha $
of edges such that  $ l_{\alpha_j} = j$ for $j = 0, \dots , p$.
 For an edge $X_\alpha$,
 we define $(\overline{\FF}_{\alpha})_\Z$  as  the free $\Z$-module generated by the elements
$\overline{F}_{{\alpha_0},\dots,{\alpha_p}=\alpha}$
 la\-bel\-ed by the elements of
the set of all flags  starting at $X_\alpha$.
We define  $(\FF_{\alpha})_\Z$ as the quotient of
$(\overline{\FF}_{\alpha})_\Z$ by the submodule generated by all
the elements of the form
\bean
\label{relations F}
{\sum}_{X_\beta,
X_{\alpha_{j-1}}\supset X_\beta\supset X_{\alpha_{j+1}}}\!\!\!
\overline {F}_{{\alpha_0},\dots,
{\alpha_{j-1}},{\beta},{\alpha_{j+1}},\dots,{\alpha_p}=\alpha}\, .
\eean
Such an element is determined  by  $j \in \{ 1, \dots , p-1\}$ and
an incomplete flag $X_{\alpha_0}\supset...\supset
X_{\alpha_{j-1}} \supset X_{\alpha_{j+1}}\supset...\supset
X_{\alpha_p} = X_\alpha$ with $l_{\alpha_i}$ $=$ $i$.

We denote by ${F}_{{\alpha_0},\dots,{\alpha_p}}$ the image in $(\FF_\alpha)_\Z$ of the element
$\overline{F}_{{\alpha_0},\dots,{\alpha_p}}$.  For $p=0,\dots,k$, we set
$({\FF}^p(\A))_\Z = \oplus_{X_\alpha,\, l_\alpha=p}\, ({\FF}_{\alpha})_\Z$,\
${\FF}^p(\A) = ({\FF}^p(\A))_\Z\otimes\C$,\
$\FF(\A) = \oplus_{p=1}^{N}\FF^p(\A)$.
We define the differential $d_\Z : (\FF^p(\A))_\Z \to (\FF^{p+1}(\A))_\Z$ by
\bean
\label{d in F}
d_\Z\  :\   {F}_{{\alpha_0},\dots,{\alpha_p}} \ \mapsto \
{\sum}_{X_{\beta},
X_{\alpha_{p}}\supset X_{\beta}}
 {F}_{{\alpha_0},\dots,
{\alpha_{p}},{\beta}} ,
\eean
$d^2_\Z=0$. Tensoring $d_\Z$  with $\C$, we obtain the differential $d : \FF^p(\A) \to \FF^{p+1}(\A)$.
In particular, we have
\bean
\label{z vs C}
H^p(\FF(\A),d) = H^p((\FF(\A))_\Z,d_\Z)\otimes \C .
\eean

\begin{thm}
[{\cite[Corollary 2.8]{SV}}]
\label{them cooh F}

We have  $H^p(\FF(\A),d)=0$ for $p\ne k$ and
\linebreak
$\dim H^k(\FF(\A),d)=|\chi(U(\A))|$, where $\chi(U(\A))$ is the Euler characteristic  of the complement $U(\A)$.
\qed
\end{thm}

\subsection{Euler characteristic of $U(\A)$}

A formula for the Euler characteristic $\chi(U(\A))$
in terms of the matroid associated with $\A$ is given in
\cite[Proposition 2.3]{STV}. The condition
$\chi(U(\A))\ne 0$ is discussed in \cite[Theorem 2]{C}, cited as Theorem 2.4 in \cite{STV}.
On the equality of the absolute value $|\chi(U(\A))|$ and the number of bounded components
of the real part of $U(\A)$ see, for example, \cite{V2}.

\subsection{Duality}
The vector spaces $\OS^p(\A)$ and $\FF^p(\A)$ are dual,  see \cite{SV}.
The pairing $ \OS^p(\A)\otimes\FF^p(\A) \to \C$ is defined as follows.
{}For $H_{j_1},...,H_{j_p}$ in general position, set
$F(H_{j_1},...,H_{j_p})=F_{{\alpha_0},\dots,{\alpha_p}}$, where
$X_{\alpha_0}=\C^k$, $X_{\alpha_1}=H_{j_1}$, \dots ,
$X_{\alpha_p}=H_{j_1} \cap \dots \cap H_{j_p}$.
Then we define $\langle (H_{j_1},...,H_{j_p}), F_{{\alpha_0},\dots,{\alpha_p}}
 \rangle = (-1)^{|\sigma|},$
if $F_{{\alpha_0},\dots,{\alpha_p}}
= F(H_{j_{\sigma(1)}},...,H_{j_{\sigma(p)}})$ for some $\sigma \in S_p$,
and $\langle (H_{j_1},...,H_{j_p}), F_{{\alpha_0},\dots,{\alpha_p}} \rangle = 0$ otherwise.

An element $F \in \FF^k(\A)$ is called  {\it singular}  if
$F$ annihilates the image of the map
$d^{(a)} :   \OS^{k-1}(\A) \to \OS^{k}(\A)$,  see \cite{V4}.
Denote by
$\on{Sing}_a\FF^k(\A) \subset \FF^k(\A)$  the subspace of all singular vectors.

\subsection{Contravariant map and form, see \cite{SV}}
 The weights  $a$ determines the {\it contravariant  map}
 \bean
 \label{S map}
  \mathcal S^{(a)} : \FF^p(\A) \to \OS^p(\A),
  \quad
  {F}_{{\alpha_0},\dots,{\alpha_p}} \mapsto
\sum  a_{j_1} \cdots a_{j_p} (H_{j_1}, \dots , H_{j_p})\,,
\eean
 where the sum is taken over all $p$-tuples $(H_{j_1},...,H_{j_p})$ such that
$H_{j_1} \supset X_{\al_1}$,\dots , $ H_{j_p}\supset X_{\alpha_p}$.
Identifying $\OS^p(\A)$ with $\FF^p(\A)^*$, we consider
the map  as a bilinear form,
$S^{(a)} : \FF^p(\A) \otimes \FF^p(\A) \to \C$.
The bilinear form is
called the {\it contravariant form}.
The contravariant form  is symmetric.
For $F_1, F_2 \in \FF^p(\A)$,
\bea
S^{(a)}(F_1,F_2) =
{\sum}_{\{j_1, \dots , j_p\} \subset J} \, a_{j_1} \cdots a_{j_p}
\, \langle (H_{j_1}, \dots , H_{j_p}), F_1 \rangle
\, \langle (H_{j_1}, \dots , H_{j_p}), F_2 \rangle ,
\eea
where the sum is over all unordered $p$-element subsets.

\begin{lem} [{\cite[Lemma 3.2.5]{SV}}]
\label{lem hom}
The contravariant map \Ref{S map} defines a homomorphism of complexes
$\mc S^{(a)} :(\FF(\A),d)\to (\OS(\A),d^{(a)})$.
\qed
\end{lem}

\subsection{Generic weights}
\label{sec rem on gener wei}

\begin{thm}[{\cite[Theorem 3.7]{SV}}]
\label{thm Shap nondeg}
If the weight $a$ is such that none of the dense edges has  weight zero, then
the contravariant form is nondegenerate. In particular, we have  an isomorphism
of complexes $\mc S :(\FF(\A),d)\to (\OS(\A),d^{(a)})$.
\qed
\end{thm}

\begin{thm} [\cite{SV, Y, OT2}]
\label{thm Y, O ,SV}
If the weight $a$ is such that none of the dense edges has  weight zero, then
 $H^p(\OS^*(\A),d^{(a)}) = 0$ for $ p \ne k$
and dim $H^k(\OS^*,d^{(a)}) = |\chi(U(\A))|$.
\qed
\end{thm}

Theorem \ref{thm Y, O ,SV} is a corollary of Lemma \ref{lem hom} and Theorems \ref{them cooh F}, \ref{thm Shap nondeg}.

\begin{cor}
\label{lem on dim of sing}
If the weight $a$ is such that none of the dense edges has  weight zero, then
the dimension of\, $\on{Sing}_a\FF^k(\A)$ equals $|\chi(U(\A))|$.
\end{cor}

Notice that none of the dense edges has  weight zero if all weights are positive.

\subsection{Differential forms}
\label{Differential forms}

For  $j\in J$,  fix defining equations $f_j=0$ for the hyperplanes $H_j$,
where $f_j = b^1_jt_1+\dots+b^k_jt_k + z_j$
with $b^i_j, z_j\in \C$.
Consider the logarithmic differential 1-form
$\omega_j = df_j/f_j$ on $\C^k$.
Let $\bar{\OS}(\A)$ be the exterior $\C$-algebra of differential forms
generated by 1 and $\omega_j$, $j\in J$.
The map ${\OS}(\A) \to \bar{\OS}(\A), \ (H_j) \mapsto \omega_j$,
is an isomorphism. We identify ${\OS}(\A)$ and $\bar{\OS}(\A)$.

For $I=\{i_1,\dots,i_k\}\subset J$, denote $d_I=d_{i_1,\dots,i_k}=\det_{i,l=1}^k(b^i_{i_l})$.
Then
\bean
\label{wedge I}
\om_{i_1}\wedge\dots\wedge \om_{i_k}=\frac{ d_{i_1,\dots,i_k}}{f_{i_1}\dots f_{i_k}}\,dt_{1}\wedge\dots\wedge dt_{k}\,.
\eean

\begin{lem}
\label{lem 1/f separate}
The functions $(1/f_j)_{j\in J}$ separate points of $U(\A)$.
\qed

\end{lem}

\subsection{Master function}
\label{master function}

The {\it master function} of the weighted arrangement $(\A, a)$ is
\bean
\label{def mast k}
\Phi_{\A,a} = {\sum}_{j\in J}\,a_j \log f_j,
\eean
a multivalued function on $U(\A)$. We have
$d\Phi_{\A,a} = \sum_{j\in J} a_j\om_j = \om^{(a)} \in \OS^1(\A)$.
Let $C_{\A,a} =\{ u\in U(\A)\ |\ \frac{\der\Phi_{\A,a}}{\der t_i}(u)=0\ \on{for}\ i=1,\dots,k\}$ be the critical set of  $\Phi_{\A,a}$.
The critical point equations can be reformulated as the equation $\om^{(a)}|_u=0$.
Notice that
\bean
\label{derivative}
\frac{\der\Phi_{\A,a}}{\der t_i} = {\sum}_{j=1}^n b^i_j\frac{a_j}{f_j}\qquad
\on{and}
\qquad
 \frac{\der\Phi_{\A,a}}{\der z_j} = \frac{a_j}{f_j}.
\eean
Define the {\it Hessian} of the master function,
$\on{Hess}_{\A,a} =
\on{det}_{i,j=1}^k \Big(\frac{\der^2\Phi_{\A,a}}{\der t_i\der t_j}   \Big)$.
A critical point $u\in C_{\A,a}$ is {\it nondegenerate} if
$\on{Hess}_{\A,a}(u)\ne 0$.
We have the formula in \cite{V4}:
\bean
\label{Hess f}
\on{Hess}_{\A,a} \,=\, (-1)^k
{\sum}_{I\subset J, |I|=k} d^2_{I}\,
{\prod}_{i\in I} \frac {a_{i}}{f^2_{i}}\, .
\eean

\subsection{Isolated critical points}
\label{Isolated critical points}

\begin{thm}[\cite{V2, OT1, Si}]
\label{thm V,OT,S}
For generic exponent $a\in(\C^\times)^n$, all critical points of
$\Phi_{\A,a}$ are nondegenerate and the number of critical points equals $|\chi(U(\A))|$.
\qed
\end{thm}

  Consider the projective space $\Pee^k$ compactifying $\C^k$. Assign
the weight $a_\infty=-\sum_{j\in J} a_j$ to the hyperplane
$H_\infty=\Pee^k-\C^k$. Denote by  $\A^\vee$
the arrangement $(H_j)_{j\in J\cup \infty}$ in $\Pee^k$.
The weighted arrangement $(\A, a)$ is called {\it unbalanced}
if the weight of any dense edge of $\A^\vee$ is nonzero, see \cite{V5}.
For example, $(\A, a)$ is unbalanced if all weights $(a_j)_{j\in J}$ are positive.
The unbalanced weights form a Zariski open subset in the space of all weights $a\in (\C^\times)^n$.

\begin{lem} [{\cite[Section 4]{V5}}]
\label{lem crit 1} If $(\A, a)$ is unbalanced, then
all critical points of $\Phi_{\A,a}$  are isolated and the sum of their Milnor numbers equals
$|\chi(U(\A))|$.
\qed

\end{lem}

\subsection {Residue}

Let  $\mc O(U(\A))$ be the algebra of regular functions on $U(\A)$  and
\linebreak
$I_{\A,a} =\langle \frac{\partial \Phi_{\A,a}} {\partial t_i }
\ |\ i=1,\dots,k\ \rangle \subset \mc O(U(\A))$
the ideal generated by the first  derivatives of $\Phi_{\A,a}$.
Let $ \OC = \mc O(U(\A))/ I_{\A,a} $ be
the algebra of functions on the critical set  and
$[\,]: \mc O(U(\A)) \to  \OC$, $f\mapsto [f]$, the projection.
The algebra
$ \OC$ is finite-dimensional, if all critical points are isolated.  In that case,
$\OC = \oplus_{u \in C_{\A,a}}  \OC_{u}$,
where $ \OC_{u}$ is the local algebra corresponding to the point $u$.

\begin{lem}
[{\cite[Lemma 2.5]{V5}}]
\label{lem f_j generate}
If the algebra  $ \OC$ is finite-dimensional,
then the elements  $[1/f_j]$, $j\in J$, generate $  \OC$.
\end{lem}

Let  $\mc R_{u} :  \OC_{u} \to \C$ be the Grothendieck residue,
\bean
\label{res map}
[f] \ \mapsto \ \frac 1{(2\pi i)^k}\,\Res_{u}
\ \frac{ f}{\prod_{i=1}^k\, \frac{\der \Phi_{\A,a}}{\der t_i}}
=\frac{1}{(2\pi i)^k}\int_{\Gamma_u}
\frac{f\ dt_1\wedge\dots\wedge dt_k}{\prod_{i=1}^k \frac{\der \Phi_{\A,a}}{\der t_i}}\ .
\eean
Here  $\Gamma_u$ is the real $k$-cycle  located in a small neighborhood of $u$
and defined by the equations
$|\frac{\der \Phi_{\A,a}}{\der t_i}|=\epsilon_i,\ i=1,\dots,k$,
where  $\epsilon_s$ are sufficiently small positive numbers. The cycle is oriented
 by the condition
$d\arg \frac{\der \Phi_{\A,a}}{\der t_1}\wedge\dots\wedge d\arg \frac{\der \Phi_{\A,a}}{\der t_k} > 0$, see \cite{GH}.

Denote by $[\Hess_{\A,a}]_u $ the image
of the Hessian in  $\OC_{u}$.  We have
\bean
\label{res of Hess}
\mc R_u : [\Hess_{\A,a}]_u  \mapsto \mu_u ,
\eean
where $\mu_u = \dim_\C  \OC_u$, the {\it Milnor number}, see  \cite{AGV}.
Define the   bilinear form on $\OC_u$,
\bean
\label{Gr form}
( [f], [g])_{u}\ =\ \mc R_u ([f][g])\ .
\eean
If  $ \OC$ is finite-dimensional,
we define the {\it residue bilinear form}  $(\,,\,)_{C_{\A,a}}$ on $\OC$  as
\bea
(\,,\,)_{C_{\A,a}} = \oplus_{u\in C_{\A,a}} (\,,\,)_{u}\,.
\eea
This form is nondegenerate, see \cite{AGV}, and $([f][g],[h])_{C_{\A,a}}=([f],[g][h])_{C_{\A,a}}$ for all $[f],[g],[h]\in \OC$. In other words,
the pair $(\OC, (\,,\,)_{C_{\A,a}})$ is a {\it Frobenius algebra}.

\subsection{Canonical element}
\label{Special}

A differential $k$-form $H \in {\OS}^k(\A)$ can be written as
\linebreak
$H  =  f_H  dt_1 \wedge \dots \wedge dt_k$,
where $f_H\in \mc O(U(\A))$.
Define a map $\on{F} : U(\A) \to \FF^k(\A)$ which sends $u\in U(\A)$ to the element $\on{F}(u)\in \FF^k(\A)$ such that
$\langle H, \on{F}(u) \rangle = f_H(u)$ for any  $H \in \OS^k(\A)$.
The map $\on{F}$ is called the {\it specialization map}, the vector $\on{F}(u)$
is called the {\it special vector} at $u$, see \cite{V4}. In the theory of quantum integrable systems special
vectors are called the {\it Bethe vectors}.

Let $(F_m)_{m\in M}$ be a basis of $\FF^k(\A)$ and $(H^m)_{m\in M} \subset \OS^k(\A)$ the dual basis.
We have $H^m = f_{H^m}dt_1\wedge\dots\wedge dt_k$ for some $f_{H^m} \in \mc O(U(\A))$.
The element
\bean
\label{spec elt}
E = {\sum}_{m\in M} f_{H^m} \otimes F_m  \ \in\ \mc O(U(\A))\otimes \FF^k(\A)
\eean
is called the {\it canonical element}.
For  $u \in U(\A)$, we have
\bean
\label{formula for v(t)}
\on{F}(u) = {\sum}_{m\in M} f_{H^m}(u) F_m .
\eean
Let $[E]$ be the image of the canonical element in $ \OC \otimes\FF^k(\A)$.

\begin{lem}
[{\cite[Lemma 2.6]{V4}}]

\label{lem v sing}
We have $[E] \in  \OC \otimes \on{Sing}_a \FF^k(\A)$.

\end{lem}

\begin{thm} [\cite{V4}]
\label{first theorem}
For  $u\in U(\A)$,  we have
\bean
\label{Shap norm}
S^{(a)}(\on{F}(u),\on{F}(u))  =  (-1)^k \Hess_{\A,a} (u).
\eean
Moreover, if $u^1, u^2 \in U(\A)$ are distinct isolated critical points of $\Phi_{\A,a}$,
then the special singular vectors $\on{F}(u^1), \on{F}(u^2)$ are orthogonal,
\bean
\label{oRt}
S^{(a)}(\on{F}(u^1),\on{F}(u^2)) = 0 ,
\eean
cf. \cite{MV,V5}.
\qed
\end{thm}

\subsection{Canonical isomorphism}

Assume that the algebra $ \OC $ is finite-dimensional. Define the  linear map
\bean
\label{map alpha}
\mc E\  : \ \OC \to \sing \FF^k(\A), \qquad
[g] \mapsto ([g],[E])_{C_{\A,a}}.
\eean

\begin{thm}
[\cite{V5}]
\label{thm alpha}
If the weight $a\in (\C^\times)^n$ is unbalanced, then the map $\mc E$ is an isomorphism of vector
spaces. The isomorphism $\mc E$ identifies the residue form on $\OC$ and
the contravariant form on $\sing \FF^k(\A)$ multiplied by $(-1)^k$,
\bean
(f,g)_{C_{\A,a}} = (-1)^kS^{(a)}(\mc E(f),\mc E(g))\qquad \text{for all}\ f,g\in \OC.
\qquad\qquad \qed
\eean

\end{thm}

The map $\mc E$ is called the  {\it canonical} isomorphism. We provide a proof of Theorem \ref{thm alpha} in Section
\ref{proofs}.

\begin{cor}[\cite{V5}]
\label{cor nondeg}

If the weight $a\in (\C^\times)^n$ is unbalanced, then the
 restriction of the contravariant form $S^{(a)}$ to the subspace
$\sing \FF^k(\A)$ is nondegenerate.
\qed
\end{cor}

On the restriction of $S^{(a)}$ to the subspace  $\sing \FF^k(\A)$ see also \cite{FaV}.

 If all critical points  are nondegenerate, then
\bean
\label{EX}
\mc E\ :\ [g]\
 \mapsto\ {\sum}_{u\in C_{\A,a}} {\sum}_m \frac {g(u)f_{H^m}(u)}{\Hess_{\A,a}(u)} F_m
= {\sum}_{u \in C_{\A,a}}  \frac { g(u)}{\Hess_{\A,a}(u)} \on{F}(u)\,,
\eean
see \Ref{res of Hess}.

\begin{rem}
If the weight $a\in(\C^\times)^n$ is unbalanced, then the
 canonical isomorphism $\mc E$ induces a commutative associative algebra structure
on $\sing_a \FF^k(\A)$. Together with the contravariant form  $S^{(a)}|_{\on{Sing}_a\FF^k}$  it is a Frobenius algebra.
The algebra of multiplication operators on $\sing_a \FF^k(\A)$ is an analog of the {\it Bethe algebra} in the theory of quantum integrable models,
see, for example, \cite{MTV1, V5}.

\end{rem}

\subsection{Orthogonal projection}
\label{Orthogonal projection}

\begin{lem}
\label{lem orth Pr}

It the weight $a\in(\C^\times)^n$ is unbalanced, then $d \FF^{k-1}(\A)= \on{Sing}_a\FF^k(\A)^\perp$,
where $d \FF^{k-1}(\A)\subset \FF^k(\A)$ is the image of the differential defined by \Ref{d in F} and
$\on{Sing}_a\FF^k(\A)^\perp \subset \FF^k(\A)$ is the orthogonal complement to $\on{Sing}_a\FF^k(\A)$
with respect to $S^{(a)}$.

\end{lem}

\begin{proof}
We have $d \FF^{k-1}(\A) \subset\on{Sing}_a\FF^k(\A)^\perp$ by Lemma \ref{lem hom} and the definition of
$\on{Sing}_a\FF^k(\A)$. But  $\dim d \FF^{k-1}(\A) = \dim \on{Sing}_a\FF^k(\A)^\perp$ by Theorem
\ref{them cooh F} and Corollary \ref{lem on dim of sing}.
\end{proof}

\begin{cor}
\label{cor ISo}
It the weight $a\in(\C^\times)^n$ is unbalanced, the orthogonal projection
 $\pi^\perp : \FF^k(\A)\to \on{Sing}_a\FF^k(\A)$ establishes the isomorphism
$H^k(\FF(\A),d) \cong \on{Sing}_a\FF^k(\A)$.
\end{cor}

Define the map
\bean
\label{psi map}
[\mc S^{(a)}]\  :\  \FF^k(\A) \to \OC,
\qquad
F \mapsto\in [f]\,,
\eean
where
$\mc S^{(a)}(F) = f dt_1\wedge\dots\wedge dt_k$.
Clearly, $ [\mc S^{(a)}](\on{Sing}_a\FF^k(\A)^\perp) = [\mc S^{(a)}](d \FF^{k-1}(\A))=0$,
since $\om^{(a)}=0$ on $C_{\A,a}$. 
In particular, $[\mc S^{(a)}]$ induces the map
\bean
\label{H^k to OC}
[\mc S^{(a)}]\  :\ H^k(\FF(\A),d)\ \to\ \OC.
\eean

\begin{thm}
\label{1st main thm}
It the weight $a\in(\C^\times)^n$ is unbalanced, then the  map
\bean
\label {OC=Sing}
[\mc S^{(a)}]\big|_{\on{Sing}_a\FF^k(\A)}\  : \ \on{Sing}_a\FF^k(\A)\ \to \ \OC
\eean
is an isomorphism of vector spaces and
\bean
\label{(-1)^k identity}
 \mc E {\circ} [\mc S^{(a)}]\big|_{\on{Sing}_a\FF^k(\A)}=(-1)^k .
 \eean

\end{thm}

Identity \Ref{(-1)^k identity} was conjectured in \cite{V5}. It was proved in  \cite{V5} that
the left-hand side in  \Ref{(-1)^k identity} is a nonzero scalar operator, if $\A$ is a generic arrangement.

\begin{rem}
The map $[\mc S^{(a)}]\big|_{\on{Sing}_a\FF^k(\A)}  :  \on{Sing}_a\FF^k(\A) \to  \OC$ is elementary.
The map $\mc E : \OC \to \on{Sing}_a\FF^k(\A)$ is  transcendental: it is given by a $k$-dimensional
integral. Formula \Ref{(-1)^k identity} says that the inverse map to the
transcendental map is  elementary.
\end{rem}

\begin{cor}
\label{1st main cor}
If the weight $a\in(\C^\times)^n$ is unbalanced, then the  map $[\mc S^{(a)}] : H^k(\FF,d) \to \OC$
is an isomorphism of vector spaces.

\end{cor}

\subsection{Proof of Theorems \ref{thm alpha} and \ref{1st main thm}}
\label{proofs}

First assume that the weight $a$ is generic and all critical points of $\Phi_{\A,a}$ are nondegenerate. Then the special
vectors $(\on{F}(u))_{u\in C_{\A,a}}$ form a basis of $\on{Sing}_a\FF^k(\A)$ by Theorems \ref{first theorem},
  \ref{thm V,OT,S} and Corollary \ref{lem on dim of sing}. (In the theory of quantum integrable systems this fact   is called the
  {\it completeness of the Bethe ansatz method}, see \cite{V3, V4}.)

Theorem \ref{first theorem} and formula \Ref{res of Hess} applied to the basis $(\on{F}(u))_{u\in C_{\A,a}}$
show that $S^{(a)}$ restricted to $\Sing_a\FF^k(\A)$ is nondegenerate and $\mc E :  \OC \to \sing \FF^k(\A)$ is an isomorphism of vector spaces
that identifies the residue form on $\OC$ and  the contravariant form on $\sing \FF^k(\A)$ multiplied by $(-1)^k$.
(More precisely, this follows from the following fact. 
Let $u \in C_{\A,a}$ and $g_u\in \OC$ be the function which equals 1 at $u$ and equals 0 at other points of $C_{\A,a}$.
Then $\mc E(g_u) = \on{F}(u)/\Hess_{\A,a}(u)$.)

The orthogonal projection $\FF^k(\A) \to \Sing_a\FF^k(\A)$ is defined by the formula
\bean
\label{proj}
F\mapsto {\sum}_{u\in C_{\A,a}} \frac{S^{(a)}(F,\on{F}(u))}{S^{(a)}(\on{F}(u),\on{F}(u))}\on{F}(u)
= (-1)^k {\sum}_{u\in C_{\A,a}} \frac{S^{(a)}(F,\on{F}(u))}{\Hess_{\A,a}(u)}\on{F}(u) .
\eean
Let $\mc S^{(a)}(F) = f dt_1\wedge\dots\wedge dt_k$ and $u\in U(\A)$, then $f(u) = S^{(a)}(F, \on{F}(u))$ by the definitions of
$\on{F}(u)$, $\mc S^{(a)}$,  $S^{(a)}$. Hence the map $[\mc S^{(a)}]$ defined in \Ref{psi map} sends $F$ to the element of
$\OC$ which equals $S^{(a)}(F, \on{F}(u))$ at every $u\in C_{\A,a}$. Applying formula \Ref{EX} to this element we obtain
\bean
\label{composition}
\mc E {\circ} [\mc S^{(a)}]\ :\ F \ \mapsto \ {\sum}_{u \in C_{\A,a}}  \frac {S^{(a)}(F, \on{F}(u))}{\Hess_{\A,a}(u)} \on{F}(u)\,.
\eean
These two formulas prove Theorem \ref{1st main thm} if all critical points are nondegenerate.

Assume now that the weight $a$ is unbalanced. Then all critical points of $\Phi_{\A,a}$ are isolated and the sum of
the corresponding Milnor numbers
equals $|\chi(U(\A))|$.  We deform the weight $a$ to make it generic and to make all critical points nondegenerate. Then
$\Sing_a\FF(\A)$ and $S^{(a)}\big|_{\Sing_a\FF(\A)}$ continuously depend on the deformation as well as
$\OC$ and $(\,,\,)_{C_{\A,a}}$.
The maps $[\mc S^{(a)}]\big|_{\Sing_a\FF(\A)}$ and $\mc E$
also continuously depend on  the deformation. This implies that
for the initial unbalanced weight $a$, we have
the identity  $ \mc E {\circ} [\mc S^{(a)}]\big|_{\on{Sing}_a\FF^k(\A)}=(-1)^k$
and the fact that $\mc E$ identifies
the residue form on $\OC$ and  the contravariant form on $\sing \FF^k(\A)$ multiplied by $(-1)^k$.
This proves Theorems \ref{1st main thm} and \ref{thm alpha}.

\subsection{Integral structure on $\OC$ and $\Sing_a\FF^k(\A)$}
If the weight $a$ is unbalanced, the formula $H^p(\FF(\A),d) = H^p((\FF(\A))_\Z,d_\Z)\otimes \C$
 and the isomorphism  $[\mc S^{(a)}]\big|_{\on{Sing}_a\FF^k(\A)} : H^k(\FF,d) \to \OC$
define an integral structure on $\OC$.
More precisely, for   a $k$-flag of edges $X_{\al_0}\supset X_{\al_1}\supset\dots\supset X_{\al_k}$, let
$\mc S^{(a)}(F_{\al_0,\dots,\al_k}) = f_{\al_0,\dots,\al_k} dt_1\wedge\dots\wedge dt_k$.
Denote by $w_{\al_0,\dots,\al_k}$ the element $[f_{\al_0,\dots,\al_k}]\in \OC$.

\begin{cor}
\label{marked span}
If the weight $a$ is unbalanced, then the
set of all elements $\{w_{\al_0,\dots,\al_k}\}$, labeled by all  $k$-flag of edges  of $\A$, spans
the vector space  $\OC$.
All linear relations between the elements of the set are corollaries of the relations
\bean
\label{relations w}
&&
{\sum}_{X_\beta,
X_{\alpha_{j-1}}\supset X_\beta\supset X_{\alpha_{j+1}}}\!\!\!
w_{{\alpha_0},\dots,
{\alpha_{j-1}},{\beta},{\alpha_{j+1}},\dots,{\alpha_p}=\alpha}=0\,,
\\
&&
{\sum}_{X_{\beta},
X_{\alpha_{p}}\supset X_{\beta}}
 w_{{\alpha_0},\dots,
{\alpha_{p}},{\beta}} = 0\,,
\notag
\eean
cf. formulas \Ref{relations F}, \Ref{d in F}.
\qed
\end{cor}

Similarly,  for   a $k$-flag of edges $X_{\al_0}\supset X_{\al_1}\supset\dots\supset X_{\al_k}$, let
$v_{\al_0,\dots,\al_k}$ be the orthogonal projection of $F_{\al_0,\dots,\al_k}$ to $\Sing_a\FF^k(\A)$.

\begin{cor}
If the weight $a$ is unbalanced, then the
set of all elements $\{v_{\al_0,\dots,\al_k}\}$, labeled by all  $k$-flag of edges  of $\A$, spans
the vector space  $\Sing_a\FF^k(\A)$.
All linear relations between the elements of the set are corollaries of the relations
\bean
\label{relations v}
&&
{\sum}_{X_\beta,
X_{\alpha_{j-1}}\supset X_\beta\supset X_{\alpha_{j+1}}}\!\!\!
v_{{\alpha_0},\dots,
{\alpha_{j-1}},{\beta},{\alpha_{j+1}},\dots,{\alpha_p}=\alpha}=0\,,
\\
&&
{\sum}_{X_{\beta},
X_{\alpha_{p}}\supset X_{\beta}}
 v_{{\alpha_0},\dots,
{\alpha_{p}},{\beta}} = 0\,,
\notag
\eean
cf. formulas \Ref{relations F}, \Ref{d in F}.
\qed
\end{cor}

We have
\bean
\label{marked correspondence}
 [\mc S^{(a)}] : v_{\al_0,\dots,\al_k} \mapsto w_{\al_0,\dots,\al_k},
\qquad
 \mc E : w_{\al_0,\dots,\al_k} \mapsto (-1)^k v_{\al_0,\dots,\al_k}.
 \eean

The elements $\{w_{\al_0,\dots,\al_k}\}\subset \OC$ and
$\{v_{\al_0,\dots,\al_k}\}\subset \Sing_a\FF^k(\A)$ will be called the {\it
marked elements}. The relations \Ref{relations w}, \Ref{relations v}
will be called the {\it marked relations}.

\begin{rem}

An interesting problem is to express  $1\in \OC$ as a linear combination of the marked elements
$w_{{\alpha_0},\dots, {\alpha_{k}}}$, see \cite{V6}, where such a formula is given for a generic arrangement.
Notice, that if all points of the critical set $C_{\A,a}$ are nondegenerate, then
\bea
\mc E(1) ={\sum}_{u\in C_{\A,a}} \on{F}(u)/\Hess_{\A,a}(u)\,,
\eea
 see \cite{MTV1}, where such sums were studied.
\end{rem}

\subsection{Skew-commutative versus commutative}
\label{versus}

If $a$ is unbalanced, then
\bean
\label{ISo}
H^k(\OS(\A),d^{(a)})\cong H^k(\FF(\A),d) \cong \OC
\eean
as vector spaces.
The first space is a cohomology space of a skew-commutative graded algebra $\OS(\A)$.
The last  space is the vector space of a commutative algebra. Isomorphisms \Ref{ISo} identify
these skew-commutative and commutative objects. It is interesting
to identify the multiplication operators on the last space with
suitable operators on the first two spaces. It turns out that those operators
appear in the associated Gauss-Manin (hypergeometric) differential equations, see Section
\ref{Critical set} and \cite{MTV1, V4, V5}.

Another identification of skew-commutative and commutative objects of an arrangement
see in \cite{GV, P}.

\subsection{Combinatorial connection}
\label{sec Comb Conn}

Consider a deformation $\A(s)$ of the  arrangement $\A$, which preserves the combinatorics of $\A$.
Assume that  the edges of $\A(s)$ can be identified with the edges of $\A$ so that
the elements in formula  \Ref{relations F} and the differential in formula \Ref{d in F}
do not depend on $s$. Assume that the deformed arrangement $\A(s)$ has a deformed weight $a(s)$, which remains
unbalanced. Then for every $s$, the elements
$\{w_{\al_0,\dots,\al_k}(s)\}$ span $\mc O(C_{\A(s),a(s)})$ as a vector space with linear relations
\Ref{relations w}  not depending on $s$. This  allows us to identify all the vector spaces $\mc O(C_{\A(s),a(s)})$.
In particular, if an element $w(s)\in \mc O(C_{\A(s),a(s)})$ is given, then the derivative $\frac{dw}{ds}$ is well-defined.
This construction is called the {\it combinatorial connection} on the family of algebras $\mc O(C_{\A(s),a(s)})$, see \cite{V6}.
All the elements $\{w_{\al_0,\dots,\al_k}(s)\}$ are flat sections of the combinatorial connection.

Similarly we can define the combinatorial connection on the family of vector spaces
\linebreak
 $\Sing_{a(s)}\FF^k(\A(s))$.

\subsection{Arrangement with normal crossings}

\label{An arrangement with normal crossings only}

 An essential arrangement $\A$ is {\it with normal crossings},
if exactly $k$ hyperplanes meet at every vertex of $\A$.
Assume that $\A$ is an essential arrangement with normal crossings.
A subset $\{j_1,\dots,j_p\}\subset J$ is called {\it independent} if
the hyperplanes $H_{j_1},\dots,H_{j_p}$ intersect transversally.

A basis of $\OS^p(\A)$ is formed by
$(H_{j_1},\dots,H_{j_p})$, where
$\{{j_1} <\dots <{j_p}\}$  are independent ordered $p$-element subsets of
$J$. The dual basis of $\FF^p(\A)$ is formed by the corresponding vectors
$F(H_{j_1},\dots,H_{j_p})$.
These bases of $\OS^p(\A)$ and $\FF^p(\A)$  are called {\it standard}.
We  have
\bean
\label{skew}
F(H_{j_1},\dots,H_{j_p}) = (-1)^{|\sigma|}
F(H_{j_{\sigma(1)}},\dots,H_{j_{\sigma(p)}}), \qquad \on{for}\ \sigma \in \Si_p.
\eean
For an independent subset $\{j_1,\dots,j_p\}$, we have
$S^{(a)}(F(H_{j_1},\dots,H_{j_p}) , F(H_{j_1},\dots,H_{j_p})) = a_{j_1}\cdots a_{j_p}$
and
$S^{(a)}(F(H_{j_1},\dots,H_{j_p}) , F(H_{i_1},\dots,H_{i_k})) = 0$
for distinct elements of the standard basis.
If $a$ is unbalanced, then the marked elements in $\OC$ are
\bean
\label{markeD}
w_{i_1,\dots,i_k} = d_{i_1,\dots,i_k} \frac{a_{i_1}}{[f_{i_1}]}\dots \frac{a_{i_k}}{[f_{i_k}]}\,,
\eean
where  $\{i_1,\dots,i_k\}$ runs through the set of all independent $k$-element subsets of $J$. We have
\bean
\label{skew marked}
w_{i_{\sigma(1)},\dots,i_{\sigma(k)}} = (-1)^\sigma w_{i_1,\dots,i_k}, \qquad \on{for}\ \sigma \in \Si_k.
\eean
We put $w_{i_1,\dots,i_k}=0$ if the set $\{i_1,\dots,i_k\}$ is dependent. The marked relations are labeled by independent subsets
$\{i_2,\dots,i_k\}$ and have the form
\bean
\label{rel Ma}
{\sum}_{j\in J} w_{j,i_2,\dots,i_k} = 0 .
\eean
The marked elements $v_{i_1,\dots,i_k}$ in $\Sing_a\FF^k(\A)$ are orthogonal projections to $\Sing_a\FF^k(\A)$ of
the elements $F(H_{i_1},\dots, H_{i_k})$ with the skew-symmetry property
\bean
\label{skew marked v}
v_{i_{\sigma(1)},\dots,i_{\sigma(k)}} = (-1)^\sigma v_{i_1,\dots,i_k}, \qquad \on{for}\ \sigma \in \Si_k.
\eean
and the marked relations
\bean
\label{rel Ma v}
{\sum}_{j\in J} v_{j,i_2,\dots,i_k} = 0
\eean
labeled by independent subsets  $\{i_2,\dots,i_k\}$.

\section{Family of parallelly transported hyperplanes}
\label{sec par trans}

\subsection{Arrangement in  $\C^n\times\C^k$}

\label{An arrangement in}

Recall that $J=\{1,\dots,n\}$.
Consider $\C^k$ with coordinates $t_1,\dots,t_k$,\
$\C^n$ with coordinates $z_1,\dots,z_n$, the projection
$\pi:\C^n\times\C^k \to \C^n$.
Fix $n$ nonzero linear functions on $\C^k$,
$g_j=b_j^1t_1+\dots + b_j^kt_k,$\ $ j\in J,$
where $b_j^i\in \C$. We assume that the functions $g_j, j\in J$,
span the dual space $(\C^k)^*$.

Define $n$ linear functions on $\C^n\times\C^k$,
$f_j = g_j + z_j = b_j^1t_1+\dots + b_j^kt_k + z_j ,$\ $ j\in J.$
We consider in $\C^n\times \C^k$
 the arrangement of hyperplanes $\A  = \{ H_j\}_{j\in J}$, where $ H_j$ is the zero set of $f_j$,
and denote by $ U(\A ) = \C^n\times \C^k - \cup_{j\in J} H_j$ the complement.

\begin{lem}
\label{lem rel for z f}
For any linear relation  $\sum_{j=1}^n \bt_jg_j=0$ we have the relation
\bean
\label{z f rela}
{\sum}_{j=1}^n \bt_j(z_j-f_j)=0.  \qquad\qquad \qquad\qquad\qquad \qed
\eean
\end{lem}

For every $x=(x_1,\dots,x_n)\in \C^n$, the arrangement $\A$
induces an arrangement $\A(x)$ in the fiber $\pi^{-1}(x)$. We
identify every fiber with $\C^k$. Then $\A(x)$ consists of
hyperplanes  $\{H_j(x)\}_{j\in J}$, defined in $\C^k$ by the equations
$g_j+x_j=0$. Thus $\{\A(x)\}_{x\in \C^n}$ is a family of arrangements in $\C^k$,
whose hyperplanes are transported parallelly as $x$ changes.
 We denote by $ U(\A(x)) = \C^k - \cup_{j\in J} H_j(x)$  the complement.

For almost all points $x\in \C^n$, the arrangement $\A(x)$ is with normal crossings.
Such points  form  the complement in $\C^n$ to the union
of suitable hyperplanes called the {\it discriminant}.

\subsection{Discriminant}
\label{Discr}

The collection $(g_j)_{j\in J}$ induces a
matroid structure on $J$.  A subset $C=\{i_1,\dots,i_r\}\subset J$ is
a {\it circuit}  if $(g_i)_{i\in C}$ are linearly dependent but any
proper subset of $C$ gives linearly independent $g_i$'s.
Denote by $\frak C$ the set of all circuits in $J$.

For a circuit $C=\{i_1,\dots,i_r\}$, \  let
$(\la^C_i)_{i\in C}$ be a nonzero collection of complex numbers such that
$\sum_{i\in C} \la^C_ig_i = 0$.
 Such a collection  is unique up to
multiplication by a nonzero number.
For every circuit $C$ we fix such a collection
and denote $f_C = \sum_{i\in C} \la^C_iz_i$.
The equation $f_C=0$ defines a hyperplane $H_C$ in
$\C^n$. It is convenient to assume that $\la^C_i=0$ for $i\in J-C$ and write
$f_C = \sum_{i\in J} \la^C_iz_i$.

\begin{lem}
\label{circ generate all}
Any linear relation $\sum_{j\in J} c_jg_j =0$ is a linear combination of  relations
\linebreak
$\sum_{i\in J} \la^C_ig_i = 0$ associated with circuits $C\in\frak C$.
\qed
\end{lem}

For any $x\in\C^n$, the hyperplanes $\{H_i(x)\}_{i\in C}$ in $\C^k$ have nonempty
intersection if and only if $x\in H_C$. If $x\in H_C$, then the
intersection has codimension $r-1$ in $\C^k$.
The union  $\Delta = \cup_{C\in \frak C} H_C$ is called the {\it discriminant}.
The arrangement $\A(x)$ in $\C^k$
has normal crossings  if and only if $x\in \C^n-\Delta$, see \cite{V5}.

On the discriminant see also \cite{BB}.

\subsection{Combinatorial connection}
\label{sec Good fibers}

For any
$x^1, x^2\in \C^n-\Delta$, the spaces $\FF^p(\A(x^1))$, $\FF^p(\A(x^2))$
 are canonically identified if a vector $F(H_{j_1}(x^1),\dots,H_{j_p}(x^1))$ of the first space
is identified  with the vector $F(H_{j_1}(x^2),\dots,H_{j_p}(x^2))$ of the second. In other words, we identify the standard bases of
these spaces.

Assume that a weight $a\in(\C^\times)^n$ is given. Then each arrangement $\A(x)$  is weighted.
The identification of spaces $\FF^p(\A(x^1))$,
$\FF^p(\A(x^2))$ for $x^1,x^2\in\C^n-\Delta$ identifies the corresponding subspaces
$\on{Sing}_a\FF^k(\A(x^1))$, $\on{Sing}_a\FF^k(\A(x^2))$ and contravariant forms.

Assume that the weighted arrangement  $(\A(x),a)$ is unbalanced  for some $x\in\C^n-\Delta$,
then $(\A(x),a)$ is unbalanced for all $x\in\C^n-\Delta$.
The identification of $\on{Sing}_a\FF^k(\A(x^1))$ and $\on{Sing}_a\FF^k(\A(x^2))$
also identifies the marked elements $v_{j_1,\dots,j_k}(x^1)$ and $v_{j_1,\dots,j_k}(x^1)$,
see Section \ref{An arrangement with normal crossings only}.
For  $x\in\C^n-\Delta$, denote $V=\FF^k(\A(x))$, $\on{Sing}_a V=\on{Sing}_a\FF^k(\A(x))$,  $v_{j_1,\dots,j_k} = v_{j_1,\dots,j_k}(x)$.
The triple $(V, \on{Sing}_a V, S^{(a)})$, with marked elements $v_{j_1,\dots,j_k}$,
 does not depend on  $x$ under the identification.

As a result of this reasoning we obtain the canonically trivialized
vector bundle
\bean
\label{combbundle}
\sqcup_{x\in \C^n-\Delta}\,\FF^k(\A(x))\to \C^n-\Delta,
\eean
with the canonically trivialized subbundle $\sqcup_{x\in \C^n-\Delta}\, \on{Sing}_a\FF^k(\A(x))\to \C^n-\Delta$
and the constant contravariant form on the fibers.
This trivialization identifies the bundle in \Ref{combbundle} with the bundle
$(\cd)\times V\to \cd$
and the subbundle  $\sqcup_{x\in \C^n-\Delta}\,\on{Sing}_a\FF^k(\A(x))\to \C^n-\Delta$ with the subbbundle
\bean
\label{comb Bundle}
(\cd)\times(\sv)\to \cd.
\eean
The bundle in \Ref{comb Bundle} will be
 called the {\it combinatorial bundle}, the flat connection on it will be called {\it combinatorial}, see \cite{V5, V6},  cf. Section
\ref{sec Comb Conn}.

\subsection{Operators $K_j\in \mc O(\C^n-\Delta)\otimes(\End {V})$, $j\in J$}
\label{sec key identity}

For a circuit $C=\{i_1, \dots, i_r\}\subset J$, we
define the linear operator $L_C : V\to V$ as follows.
Let $C_m=C-\{i_m\}$. Let $F(H_{j_1},\dots,H_{j_k})$ be
an element of the standard basis.
We set $L_C : F(H_{j_1},\dots,H_{j_k}) \mapsto 0$ if
$|\{{j_1},\dots,{j_k}\}\cap C| < r-1$.
If $\{{j_1},\dots,{j_k}\}\cap C = C_m$, then by \Ref{skew}
we have
$
F(H_{j_1},\dots,H_{j_k})
=
\pm F(H_{i_1},H_{i_2},\dots,\widehat{H_{i_{m}}},\dots,H_{i_{r-1}}, H_{i_{r}},H_{s_1},\dots,H_{s_{k-r+1}})
$
with $\{{s_1},\dots,{s_{k-r+1}}\}=
\{{j_1},\dots,{j_k}\}-C_m$.
We set
\bean
\label{L_C}
&&
L_C :
F(H_{i_1},\dots,\widehat{H_{i_{m}}},\dots,H_{i_{r}},H_{s_1},\dots,H_{s_{k-r+1}})
 \mapsto
\\
\notag
&&
\phantom{aaaaaaaaa}
(-1)^m {\sum}_{l=1}^{r} (-1)^l a_{i_l}
F(H_{i_1},\dots,\widehat{H_{i_{l}}},\dots,H_{i_{r}},H_{s_1},\dots,H_{s_{k-r+1}}) .
\eean
Consider on $\C^n\times\C^k$ the logarithmic 1-forms
$\omega_C = \frac {df_C}{f_C}, \ C\in \frak C$.
Recall  $f_C = \sum_{j\in J}\la^C_jz_j$.
We set
\bean
\label{K_j}
K_j \ = \ {\sum}_{C\in \frak C}
\,\frac{\la_j^C}{f_C} \,L_C \ {}\ \in \mc O(\C^n-\Delta)\otimes(\End {V})  .
\eean
We have
\bean
\label{LK}
{\sum}_{C\in \frak C}
\omega_{C} \otimes L_C = {\sum}_{j\in J} dz_j\otimes K_j .
\eean

\begin{thm} [\cite{V5}]
\label{thm K sym}
For any $j\in J$ and $x\in \C^n-\Delta$, the operator $K_j(x)$ preserves
the subspace $\on{Sing}_a V\subset V$ and is a symmetric operator,
$S^{(a)}(K_j(x)v,w)= S^{(a)}(v, K_j(x)w)$ for all $v,w\in V$.
\end{thm}

\subsection{Corollary of Theorem \ref{thm K sym}}

We obtain formulas for the action of $K_j$ on the marked elements
$v_{j_1,\dots,j_k}\in\Sing_a V$ from formulas for the action of $L_C$.

Let $C=\{i_1, \dots, i_r\}$ be a circuit and $v_{j_1,\dots,j_k}\in\Sing_a V$ a marked element.
If $|\{{j_1},\dots,{j_k}\}\cap C| < r-1$, then $L_C(v_{j_1,\dots,j_k})=0$.
If $\{{j_1},\dots,{j_k}\}\cap C = C_m$, then by   \Ref{skew marked v}
we have
$v_{j_1,\dots,j_k}
=
\pm v_{i_1,i_2,\dots,\widehat{i_{m}},\dots,i_{r-1},i_{r},s_1,\dots,s_{k-r+1}}$
with $\{{s_1},\dots,{s_{k-r+1}}\}= \{{j_1},\dots,{j_k}\}-C_m$.
We have
\bean
\label{L_C v}
L_C  : v_{i_1,i_2,\dots,\widehat{i_{m}},\dots,i_{r-1},i_{r},s_1,\dots,s_{k-r+1}}
\mapsto
(-1)^m {\sum}_{l=1}^{r} (-1)^l a_{i_l}
v_{i_1,\dots,\widehat{i_{l}},\dots,i_{r},s_1,\dots,s_{k-r+1}} .
\eean

\subsection{Gauss-Manin connection on  $(\cd)\times (\sv)\to \cd$}
\label{Construction}

The {\it master function} of  $(\A, a)$ is
$\Phi_{\A,a} = {\sum}_{j\in J}\,a_j \log f_j$,
a multivalued function on $U( {\A})$. Let $\kappa\in\C^\times$.
The function $e^{\Phi_{\A,a}/\kappa}$
defines a rank one local system $\mc L_\kappa$
on $U(\A)$ whose horizontal sections
over open subsets of $U(\A)$
 are univalued branches of $e^{\Phi_{\A,a}/\kappa}$ multiplied by complex numbers,
see, for example, \cite{SV, V2}.
The vector bundle
\bean
\sqcup_{x\in \C^n-\Delta}\,H_k(U(\A(x)), \mc L_\kappa\vert_{U(\A(x))})
 \to  \C^n-\Delta
 \eean
is called the {\it homology bundle}. The homology bundle has a canonical  flat Gauss-Manin connection.

For a fixed $x$, choose any $\gamma\in H_k(U(\A(x)), \mc L_\kappa\vert_{U(\A(x))})$.
The linear map
\bean
\label{int map}
\{\gamma\} \ {}:\ {} \OS^k(\A(x)) \to \C,  \qquad \om \mapsto \int_{\gamma} e^{\Phi_{\A,a}/\kappa} \om,
\eean
is an element of $\Sing\FF^k(\A(x))$ by Stokes' theorem.
It is known that for generic  $\kappa$
any element of $\Sing\FF^k(\A(x))$ corresponds to a certain
$\gamma$ and  in that case this construction  gives the {\it integration isomorphism}
\bean
\label{iSO}
H_k(U(\A(x)), \mc L_\kappa\vert_{U(\A(x))}) \to \on{Sing}_a\FF^k(\A(x)),
\eean
 see \cite{SV}. The precise values of $\kappa$, such that \Ref{iSO} is an isomorphism, can be deduced
 from the determinant formula in \cite{V1}.

For generic $\kappa$ the fiber isomorphisms \Ref{iSO} define an isomorphism of the homology bundle and
the combinatorial bundle \Ref{comb Bundle}. The  Gauss-Manin connection induces a  connection on the combinatorial bundle.
That connection on the combinatorial bundle will  also be  called the {\it Gauss-Manin connection}.

Thus, there are two connections on the combinatorial bundle: the combinatorial connection and the Gauss-Manin
connection depending on  $\kappa$. In this situation we  consider the differential equations
for flat sections of the Gauss-Manin connection with respect to the combinatorially flat standard basis. Namely,
let $\gamma(x) \in H_k(U(\A(x)), \mc L_\kappa\vert_{U(\A(x))})$ be a flat section of the Gauss-Manin connection.
Let us write the corresponding section $I_\gamma(x)$ of the bundle $(\cd)\times \Sing_a V\to \cd$
in the combinatorially flat standard basis,
\bean
\label{Ig}
&&
{}
\\
\notag
&&
I_\gamma(x) =\!\!\!
\sum_{{\rm independent } \atop \{j_1 < \dots < j_k\} \subset J }\!\!\!
I_\gamma^{j_1,\dots,j_k}(x)
 F(H_{j_1}, \dots , H_{j_k}),
 \quad
 I_\gamma^{j_1,\dots,j_k}(x)
=
\int_{\gamma(x)} e^{\Phi_{\A,a}/\kappa}
 \omega_{j_1} \wedge \dots \wedge \omega_{j_k}.
\eean
By Theorem \ref{thm K sym}, we can also write
\bean
\label{Ig v}
I_\gamma(x) =
{\sum}_{{\rm independent } \atop \{j_1 < \dots < j_k\} \subset J }
I_\gamma^{j_1,\dots,j_k}(x)
v_{j_1, \dots , j_k}.
\eean
For $I=\sum I^{j_1,\dots,j_k}v_{j_1, \dots , j_k}$ and $j\in J$, we denote
$\frac{\der I}{\der z_j} = \sum \frac{\der I^{j_1,\dots,j_k}}{\der z_j}v_{j_1, \dots , j_k}$.
This formula defines the combinatorial connection on the combinatorial bundle.

\begin{thm} [\cite{V2, V5}]
\label{thm ham normal}   The  section $I_\ga$ satisfies the differential equations
\bean
\label{dif eqn}
\kappa \frac{\der I}{\der z_j}(x) = K_j(x)I(x),
\qquad
j\in J,
\eean
where  $K_j(x)$ are the linear operators defined in \Ref{K_j}.
\qed
\end{thm}

On the Gauss-Manin connection and these differential equations see also \cite{CO}.

\subsection{Critical set}
\label{Critical set}

Denote by $C_{\A,a}$ the critical set of $\Phi_{\A,a}$ in the $\C^k$-direction,
\bean
\label{crit C new}
C_{\A,a} =\{ (x,u)\in U(\A)\subset \C^n\times\C^k\ \big|\ \frac{\der\Phi_{\A,a}}{\der t_i}(x,u)=0\ \on{for}\ i=1,\dots,k\}.
\eean

\begin{lem}
\label{lem crit smooth}
If  $C_{\A,a}$ is nonempty, then it is a smooth $n$-dimensional subvariety of $U(\A)$.
\end{lem}

\begin{proof} For $j_1,\dots,j_k\in J$, we have
\bea
\on{det}_{i,l=1}^k \big( \frac{\der^2\Phi_{\A,a}}{\der z_{j_l}\der t_i}\big) = (-1)^k\on{det}_{i,l=1}^k (b^i_{j_l})\prod_{l=1}^k \frac{a_{j_l}}{f_{j_l}^2}.
\eea
Since $(g_j)_{j\in J}$ span $(\C^k)^*$, there exists $j_1,\dots,j_k\in J$ such that $ \on{det}_{i,l=1}^k (b^i_{j_l})\ne 0$.
\end{proof}

\begin{lem}
\label{lem finite}
If $a\in(\C^\times)^n$ is generic, then
\begin{enumerate}
\item[(i)]
 every fiber of the projection
$\pi|_{C_{\A,a}}:C_{\A,a}\to \C^n$ is finite;
\item[(ii)]
for any $x\in \C^n$, the number of points of $C_{\A,a}$ in the fiber over $x$, counted with their Milnor numbers, equals $|\chi(U(\A(x)))|$;

\item[(iii)]
for generic $x\in \C^n$, each of the points of $C_{\A,a}$ in the fiber over $x$ is nondegenerate.

\end{enumerate}

\end{lem}

\begin{proof}
 The lemma follows from Theorem \ref{thm V,OT,S} and Lemma \ref{lem crit 1}.
 \end{proof}

Let $\mc O(C_{\A,a})$  be the algebra of regular functions on $C_{\A,a}$ and
$\mc O(C_{\A(x),a})$ the algebra of regular functions on $C_{\A(x),a} = C_{\A,a}\cap \pi^{-1}(x)$. Namely,
for $x=(x_1,\dots,x_n)\in\C^n$, let $I_{\A(x),a}$ be the ideal in $\mc O(U(\A(x)))$ generated by
$\frac{\der\Phi_{\A,a}}{\der t_i}, i=1,\dots,k$. We set
\bean
\label{OCx}
\mc O(C_{\A(x),a}) = \mc O(U(\A(x)))/I_{\A(x),a}\, .
\eean
Assume that the weight $a$ is such that the pair $(\A(x),a)$ is unbalanced for some $x\in\cd$. Then
$\dim \OCx = |\chi(U(\A(x)))|$ for every $x\in\cd$ and we obtain the  vector {\it bundle of algebras}
\bean
\label{algebra bundle}
\sqcup_{x\in\C^n-\Delta} \OCx \to \C^n-\Delta\,.
\eean
For $x\in\cd$, consider  the canonical element $E(x)$ of the arrangement $\A(x)$ and its
image $[E(x)]$ in $\OCx\otimes\sv$, see Lemma \ref{lem v sing}.
Recall the canonical isomorphism  \Ref{map alpha},
\bean
\label{fiber iso}
\mc E(x)  : \OCx \to \sv .
\eean
This fiber isomorphism establishes an isomorphism $\mc E$ of the bundles \Ref{algebra bundle} and  \Ref{comb Bundle}.
The isomorphism $\mc E$ and the combinatorial and Gauss-Manin connections on the bundle \Ref{comb Bundle}
induce two connections on the bundle \Ref{algebra bundle} which will also be called  the {\it combinatorial and Gauss-Manin
connections}, respectively.

\smallskip

\begin{thm}[\cite{V5}]
\label{K/f}
If  the pair $(\A(x),a)$ is unbalanced for $x\in\cd$, then for all $j\in J$, we have
\bean
\mc E(x) \circ \Big[\frac{a_j}{f_j}\Big] *_x  = K_j(x) \circ \mc E(x)\,,
\eean
where $\big[\frac{a_j}{f_j}\big] *_x$ is the operator of multiplication by
$\big[\frac{a_j}{f_j}\big]$ in $\OCx$ and $ K_j(x) : \sing_aV\to\sing_aV$ is the operator defined
in \Ref{K_j}.
\qed
\end{thm}

\begin{rem}
Recall that ${a_j/f_j} = {\der\Phi_{\A,a}/\der z_j}$ and the elements $ [{a_j/f_j}]$, $j\in J$, generate the algebra $\OCx$.
Theorem \ref{K/f} says that under the isomorphism $\mc E(x)$ the operators of multiplication $[{a_j/f_j}]*_x$
in $\OCx$ are identified with the operators
$K_j(x)$ in the Gauss-Manin differential equations \Ref{dif eqn}, cf. Section \ref{versus}. The correspondence of Theorem \ref{K/f}
defines
 a commutative algebra structure on $\sing_aV$, the structure depending on $x$. The multiplication in this commutative algebra is generated by
 the operators $K_j(x), j\in J$.  The correspondence of Theorem \ref{K/f}
  also defines the Gauss-Manin  differential equations on the bundle of algebras
in terms of the multiplication in the fiber algebras.

Notice that the relations between the operators $K_j(x)$ coincide with the relations among the elements $[a_j/f_j]$ in $C_{\A(x), a}$.

\end{rem}

\subsection{Formulas for multiplication}
\label{Formulas for multiplication}

For $x\in\cd$ and a circuit $C=\{i_1, \dots, i_r\}\subset J$, we
define the linear operator $L_C : \OCx\to \OCx$ as follows.
Let $w_{j_1,\dots,j_k}\in\OCx$ be a marked element.
If $|\{{j_1},\dots,{j_k}\}\cap C| < r-1$, then we set $L_C(w_{j_1,\dots,j_k})=0$.
If $\{{j_1},\dots,{j_k}\}\cap C = C_m$, then by \Ref{skew marked}
we have
$w_{j_1,\dots,j_k}=$
$\pm\, w_{i_1,i_2,\dots,\widehat{i_{m}},\dots,i_{r-1},i_{r},s_1,\dots,s_{k-r+1}}$
with $\{{s_1},\dots,{s_{k-r+1}}\}= \{{j_1},\dots,{j_k}\}-C_m$.
We set
\bean
\label{L_C W}
\quad
L_C (w_{i_1,i_2,\dots,\widehat{i_{m}},\dots,i_{r-1},i_{r},s_1,\dots,s_{k-r+1}}) =
(-1)^m {\sum}_{l=1}^{r} (-1)^l a_{i_l}
w_{i_1,\dots,\widehat{i_{l}},\dots,i_{r},s_1,\dots,s_{k-r+1}} ,
\eean
cf. formula \Ref{L_C v}. For $j\in J$, we define the operator $K_j(x) : \OCx\to \OCx$ by the formula
\bean
\label{K_j on w}
K_j(x) \ = \ {\sum}_{C\in \frak C}
\,\frac{\la_j^C}{f_C(x)} \,L_C ,
\eean
cf. formula \Ref{K_j}.

\begin{cor}
\label{cor *_x K_j}
If  the pair $(\A(x),a)$ is unbalanced for $x\in\cd$, then the operator of multiplication
$\big[\frac{a_j}{f_j}\big]*_x$  in $\OCx$ equals the operator
$K_j(x):\OCx\to\OCx$
defined  in \Ref{K_j on w}.
\qed
\end{cor}

\subsection{Corollary of Theorem \ref{K/f}}
\label{GM on algebras}

For a section $I=\sum_{j_1,\dots,j_k} I^{j_1,\dots,j_k}w_{j_1, \dots , j_k}$ of the bundle of algebras \Ref{algebra bundle}
 and $j\in J$, we define
$\frac{\der I}{\der z_j} = \sum \frac{\der I^{j_1,\dots,j_k}}{\der z_j}w_{j_1, \dots , j_k}$.
This formula defines the combinatorial connection on the bundle of algebras \Ref{algebra bundle}.

\begin{thm} [\cite{V6}]
\label{thm Frob}
 If a  section $I$ of the bundle of algebras \Ref{algebra bundle} is flat with respect to the Gauss-Manin connection,
then it satisfies the differential equations
\bean
\label{dif eqn crit set}
\kappa \frac{\der I}{\der z_j}(x) = \Big[\frac{a_j}{f_j}\Big]*_x I(x),
\qquad
j\in J.
\qquad\qquad\qquad \qed
\eean

\end{thm}

Notice that solutions of these differential equations are given by the hypergeometric integrals
\bean
\label{Ig w}
I_\gamma(x) =
{\sum}_{{\rm independent } \atop \{j_1 < \dots < j_k\} \subset J }
I_\gamma^{j_1,\dots,j_k}(x)\,
w_{j_1, \dots , j_k},
\eean
where $\gamma(x) \in H_k(U(\A(x)), \mc L_\kappa\vert_{U(\A(x))})$ is a flat section of the Gauss-Manin connection
on the homology bundle and
$ I_\gamma^{j_1,\dots,j_k}(x) =
\int_{\gamma(x)} e^{\Phi_{\A,a}/\kappa}
 \omega_{j_1} \wedge \dots \wedge \omega_{j_k}$.

Notice the similarities between the differential equations in \Ref{dif eqn crit set}
and the standard differential equations associated with Frobenius structures, see \cite{D, M}.

\section{Langrangian variety and critical set}
\label{sec 1}

\subsection{Lagrangian variety}
\label{sec def of Lag}

Consider  $\C^n$ with  coordinates $q_1,\dots,q_n$ and the dual space $(\C^n)^*$
with the dual coordinates $p_1,\dots, p_n$.
The space $\C^n\times \Cs$ has symplectic form
\linebreak
$\om=\sum_{j=1}^n dp_j\wedge dq_j$.
Two functions $M,N$ on $\C^n\times \Cs$ are in involution if
$\{M,N\} = $
\linebreak
$\sum_{j=1}^n \Big(\frac{\der M}{\der q_j}\frac{\der N}{\der p_j}
    - \frac{\der M}{\der p_j}\frac{\der N}{\der q_j}\Big) = 0$.

For a $k$-dimensional  vector subspace $Y\subset \C^n$, denote by $Y^{\perp}\subset \Cs$ the annihilator of $Y$.
The $n$-dimensional vector space $Y\times Y^\perp$  is a Lagrangian  subspace of $\C^n\times \Cs$ with defining equations
\bean
\label{lin eq new}
F_\al
&:=& \sum_{j=1}^n \al_jp_j =0,
\qquad \al=(\al_1,\dots,\al_n)\in Y,
\\
G_\bt
&:=& \sum_{j=1}^n \bt_j q_j =0,
\qquad
\bt=(\bt_1,\dots,\bt_n)\in Y^\perp.
\notag
\eean
The set of all functions $\{F_\al, G_\bt\}$ is in involution.

 Fix a weight  $a \in (\C^\times)^n$.
Consider the invertible rational symplectic map $r_{a} : \C^n\times\Cs\to \C^n\times\Cs$,
\bea
(q_1,\dots,q_n, p_1,\dots,p_n)\mapsto
(q_1+a_1/p_1,\dots,q_n+a_n/p_n, p_1,\dots,p_n).
\eea
Denote
\bean
\label{def L}
L_{Y,a} = r_{a}(Y\times Y^\perp) \subset \C^n\times\Cs.
\eean

\begin{lem}
\label{lem L is irr new}

Assume that for  every $j\in J$, the subspace $Y$ does not lie in the hyperplane $q_j=0$
and the subspace $Y^\perp$ does not lie in the hyperplane  $p_j=0$, then
$L_{Y, a}$  is  an irreducible smooth $n$-dimensional Lagrangian subvariety in
$\C^n\times\{ y\in \Cs \ |\ \prod_{j=1}^n p_j(y)\ne 0\}$
defined by equations
\bean
\label{eq for L new}
F_\al
&:=& \sum_{j=1}^n \al_jp_j =0,
\qquad
\phantom{phantoaa}
 \al=(\al_1,\dots,\al_n)\in Y,
\\
G_{\bt,a}
&:=& \sum_{j=1}^n \bt_j (q_j-a_j/p_j) =0,
\qquad \bt=(\bt_1,\dots,\bt_n)\in Y^\perp .
\notag
\eean
The set of all functions $\{F_\al, G_{\bt,a}\}$ is in involution.
\qed
\end{lem}

Let $I=\{i_1,\dots,i_k\}\subset J$ be a $k$-element subset and $\bar I$ its complement.

\begin{lem}
\label{I-coordinates lem new}
Under hypotheses of Lemma \ref{lem L is irr new}, assume that $I$ is such that the functions
$q_I=\{q_i\,|\, i\in I\}$ form a coordinate system on $Y$. Then the functions
$q_I$ and  $p_{\bar I}=\{p_j\,|\, j\in \bar I\}$ form a coordinate system  on $L_{Y,a}$.
\end{lem}

\begin{proof}
The functions $p_I$ are expressed in terms of $p_{\bar I}$ with the help of equations
$F_\al=0$. The functions $q_{\bar I}$ are expressed in terms of $q_I, p_{\bar I}$ with the
help of equations $G_{\bt,a}=0$.
Clearly the functions $q_I, p_{\bar I}$ are independent.
\end{proof}

We order the functions of the coordinate system $q_I,p_{\bar I}$ according to the increase of the index. For example, if $k=3, n=6$,
$I=\{1,3,6\}$, then the order is $q_1,p_2,q_3,p_4,p_5,q_6$.

Fix a basis $b^i=(b^i_1,\dots,b^i_n)$, $i=1,\dots,k$, of $Y$. Let $t=(t_1,\dots,t_k)$ be
the associated coordinate system on $Y$. Then
$q_j|_Y = \sum_{i=1}^k b^i_jt_i$\,.
For $I=\{i_1,\dots,i_k\}\subset J$, we denote $d_I=d_{i_1,\dots,i_k}=\det_{i,l=1}^k(b^i_{i_l})$,
cf. Section \ref{Differential forms}.

\begin{lem}
\label{lem trans fns new}
Let $I=\{i_1,\dots,i_k\}$ and $I'=\{i_1',\dots,i_k'\}$ be  subsets of $J$ each satisfying the hypotheses of Lemma \ref{I-coordinates lem new}.
Consider the corresponding ordered coordinate systems
  $q_I$, $p_{\bar I}$ and $q_{I'}$, $p_{\bar {I'}}$ on $L_{Y,a}$. Express the coordinates of the second system in terms of
coordinates of the first system and denote by $\on{Jac}_{I,\bar{I'}}(q_I,p_{\bar I})$ the Jacobian of this change. Then
\bea
\on{Jac}_{I,\bar{I'}}(q_I,p_{\bar I}) = (d_{i_1',\dots,i_k'}/d_{i_1,\dots,i_k})^2.
\eea
\end{lem}
\begin{proof}
This statement is proved in \cite[Lemma 5.4]{V6} under the assumption that $Y\subset \C^n$ is a generic subspace with respect to
the coordinate system $q_1,\dots,q_n$. This implies the lemma since the left-  and right-hand sides of the formula continuously depend on $Y$.
\end{proof}

Let $I=\{i_1,\dots,i_k\}\subset J$ satisfy the hypotheses of Lemma \ref{lem trans fns new}, and $q_I,p_{\bar I}$ the corresponding ordered coordinate
system on $L_{Y,a}$. The functions $q_1,\dots,q_n$ form an ordered coordinate system on $\C^n$.
Let $\pi_{L_{Y,a}}  :  L_{Y,a} \to \C^n$ be the restriction to $L_{Y,a}$ of the natural projection $\C^n\times\Cs\to\C^n$.
Let  $\on{Jac}_I(q_I,p_{\bar I})$ be the Jacobian of $\pi_{L_{Y,a}}$  with respect to the chosen coordinate systems.

\begin{thm}
\label{thm hess jac old new}
The function $d_I^2\on{Jac}_I$ on $L_{Y,a}$
 does not depend on the choice of $I$ and
\bean
\label{jac hess New}
d_I^2\on{Jac}_I =
 (-1)^{n-k}
 \sum_{M\subset J,\, |M|=n-k} d^2_{\bar M} \prod_{j\in M} \frac{a_j}{p_j^2} .
\eean

\end{thm}

\begin{proof}
The function $d_I^2\on{Jac}_I$ does not depend on $I$ by Lemma \ref {lem trans fns new}.
Formula \Ref{jac hess New} is proved in \cite[Theorem 3.8]{V6}  under the assumption that $Y\subset \C^n$ is a generic subspace with respect to
the coordinate system $q_1,\dots,q_n$. This implies the theorem since the left-  and right-hand sides of the formula continuously depend on $Y$.
\end{proof}

\subsection{Fibers of $\pi_{L_{Y,a}}$}
 For $x\in\C^n$, denote $C_{Y,a}(x) = \pi^{-1}_{L_{Y,a}}(x)$,
the fiber of the projection $\pi_{L_{Y,a}}$. The fiber is defined in $(\C^\times)^n$
with coordinates $p_1,\dots,p_n$ by the equations
\bean
\label{eq for L and gamma}
\sum_{j=1}^n \al_jp_j &=& 0,
\qquad
 \al=(\al_1,\dots,\al_n)\in Y,
\\
 \sum_{j=1}^n \bt_j (x_j-a_j/p_j)& =&0,
\qquad \bt=(\bt_1,\dots,\bt_n)\in Y^\perp ,
\notag
\eean
cf. \Ref{eq for L new}. Let $I_{L_Y(x),a}$ be the ideal in $\mc O((\C^\times)^n)$ generated by
the left-hand sides of equations \Ref{eq for L and gamma}.  Set
\bean
\label{Lx}
\mc O(L_{Y,a}(x)) = \mc O((\C^\times)^n)/I_{L_Y(x),a}   \, .
\eean

\subsection{Arrangement in  $\C^n\times\C^k$}
\label{An arrangement second part}

Return to the objects and notations of Section \ref{sec par trans} and consider $\C^k$ with coordinates $t_1,\dots,t_k$,\
$\C^n$ with coordinates $z_1,\dots,z_n$, the projection
$\pi:\C^n\times\C^k \to \C^n$ and  $n$ nonzero linear functions on $\C^k$,
$g_j=b_j^1t_1+\dots + b_j^kt_k,$\ $ j\in J$. As in Section \ref{sec par trans} we
 assume that  $g_j, j\in J$,
span the dual space $(\C^k)^*$. We consider the linear functions
$f_j = g_j + z_j = b_j^1t_1+\dots + b_j^kt_k + z_j$
on $\C^n\times\C^k$ and the arrangement of hyperplanes $\A  = \{ H_j\}_{j\in J}$ in
$\C^n\times \C^k$, where $H_j$ is defined by the equation $f_j=0$. We assume that a weight $a\in(\C^n)^\times$ is given and
consider the critical set
$C_{\A, a}$ defined by \Ref{crit C new}.

\smallskip

In the rest of the paper we denote by
\bean
\label{def Y}
     Y=Y(\A)
\eean
 { the $k$-dimensional subspace of $\C^n$ spanned by
the vectors} $b^i=(b^i_1, \dots, b^i_n)$, $i=1,\dots,k$.

\begin{thm}
\label{thm 1}
Assume that for  any $j\in J$, the subspace $ Y^\perp$ does not lie in the hyperplane $p_j=0$. Assume that
the critical set $C_{\A,a}$ is nonempty. Then the map
 \bean
 \label{Psi}
 \Psi_{\A,a}: U(\A)\to \C^n\times \Cs,
 \qquad
 (x,u)\mapsto (x,y(x,u)),
 \eean
  where
\bean
\label{iso map}
y(x,u) =  \Big( \frac{\der\Phi_{\A,a}}{\der z_1}(x,u),\dots, \frac{\der\Phi_{\A,a}}{\der z_n}(x,u) \Big)
=\Big(\frac{a_1}{f_1(x,u)},\dots,\frac{a_n}{f_n(x,u)}\Big),
\eean
restricted to $C_{\A,a}$ is a diffeomorphism of the critical set $C_{\A,a}$ onto the Lagrangian variety $L_{Y,a}$.
\end{thm}

\begin{proof}

\begin{lem}
\label{lem in} We have $\Psi_{\A,a}(C_{\A,a})\subset L_{Y,a}$.
\end{lem}

\begin{proof} If $\al\in Y$ and $(x,u)\in C_{\A,a}$, then the equation
$F_\al( \frac{\der\Phi_{\A,a}}{\der z_1}(x,u), \dots, \frac{\der\Phi_{\A,a}}{\der z_n}(x,u))=0$ is a linear combination
of the equations $ \frac{\der\Phi_{\A,a}}{\der t_i}(x,u)=0$, $i=1,\dots,k$, see \Ref{derivative}.
If $\bt\in Y^\perp$  and $(x,u)\in C_{\A,a}$, then the equation
$G_{\bt,a}(x, \frac{\der\Phi_{\A,a}}{\der z_1}(x,u), \dots, \frac{\der\Phi_{\A,a}}{\der z_n}(x,u))=0$ is just
the equation \Ref{z f rela}.
\end{proof}

\begin{lem}
\label{lem inject}
The map $\Psi_{\A,a}$ sends distinct points of  $U(\A)$ to distinct points of $\C^n\times \Cs$.
\end{lem}

\begin{proof}
It is enough  to check that $\Psi_{\A,a}(x,u)\ne \Psi_{\A,a}(x,u')$ if $u\ne u'$, but this follows from
Lemma \ref{lem 1/f separate}.
\end{proof}

\begin{lem}
\label{lem diff}
The Jacobian of the map $\Psi_{\A,a}|_{C_{\A,a}} : C_{\A,a}\to L_{Y,a}$ is never zero.

\end{lem}

\begin{proof}
The lemma follows from Lemmas \ref{lem L is irr new}, \ref{lem crit smooth}, \ref{lem inject} or by direct calculation.
\end{proof}

\begin{lem}
\label{lem onto}
We have $\Psi_{\A,a}(C_{\A,a})=L_{Y,a}$.
\end{lem}

\begin{proof}
Let $\Psi_{\A,a}(C_{\A,a})\ne L_{Y,a}$ and  $(x^0,y^0)\in L_{Y,a} - \Psi_{\A,a}(C_{\A,a})$, where $x^0\in \C^n$ and $y^0\in \Cs$.
We have $\dim (L_{Y,a} - \Psi_{\A,a}(C_{\A,a}))< n$
by Lemmas \ref{lem L is irr new}, \ref{lem crit smooth}, \ref{lem diff}.
Hence there exists a germ of an analytic curve
$\iota : (\C,0) \to  (L_{Y,a},(x^0,y^0))$ such that $\iota(s)\in \Psi_{\A,a}(C_{\A,a})$ for $s\ne 0$.
 Consider the curve $s\mapsto (\Psi_{\A,a})^{-1}(\iota(s))$
for $s\ne 0$. Let $(\Psi_{Y,a})^{-1}(\iota(s))=(x(s),u(s))$, where
 $x(s)\in \C^n$ and $u(s)\in \C^k$.   We have
$\lim_{s\to 0} \ga(s)=x^0$. Since for any $j\in J$, the function
 $\frac{\der\Phi_{\A,a}}{\der z_j}(x(s),y(s))$
has a finite limit, there is a finite nonzero limit
$u^0:=\lim_{s\to 0} u(s)$ in $U(\A(x^0))$. Clearly $(x^0,u^0)\in C_{\A,a}$
and $\Psi_{\A,a} : (x^0,u^0)\mapsto (x^0,y^0)$. We get a contradiction which proves the lemma.

One can also check the statement by direct calculation.
\end{proof}

Lemmas \ref{lem in}-\ref{lem onto} prove Theorem \ref{thm 1}.
\end{proof}

\subsection{Hessian and Jacobian}
\label{H J}

Recall the Hessian  of the master function $\Phi_{\A,a}$, see Section \ref{master function}.
 Under the diffeomorphism $C_{\A,a}\to L_{Y,a}$ of Theorem \ref{thm 1},
we may consider  the Hessian as a function on $L_{Y,a}$.
Then
\bean
\label{Hess}
\on{Hess}_{\A,a} \,=\, (-1)^k
 \sum_{I\subset J, |I|=k} \! d^2_I\, {\prod}_{i\in I}^k \frac{p_{i}^2} {a_{i}} .
\eean

\begin{cor}
\label{cor hess jac} Let $I\subset J$  satisfy the hypotheses of Lemma \ref{I-coordinates lem new}. Then
\bean
\label{jac hess}
 \Hess_{\A,a} = (-1)^n d_I^2\on{Jac}_I   {\prod}_{j\in J} \frac {p_j^2}{a_j}  \, .
\eean
\end{cor}

\begin{proof}
Formula \Ref{jac hess} follows from formulas   \Ref{jac hess New}  and  \Ref{Hess}.
\end{proof}

Notice that the ratio of $\Hess_{\A,a}$ and $d_I^2\on{Jac}_I$  is never zero.

\subsection{Corollaries of Theorem \ref{thm 1}}
\label{cor 4.5}

The map $\Psi_{\A,a}$ of
Theorem \ref{thm 1} establishes an isomorphism
\bean
\label{iso OL C}
\Psi_{\A,a}^*\  :\  \mc O(L_{Y,a}) \to \OC, \qquad [q_j],  \mapsto [z_j],
\qquad [p_j]\mapsto \Big[\frac{\der\Phi_{\A,a}}{\der z_j}\Big],
\eean
for all $j\in J$, and for any $x\in \C^n$, the isomorphism
\bean
\label{iso C L x}
\Psi_{\A(x),a}^*\  :\  \OLx \to \OCx,
\qquad [p_j]\mapsto \Big[\frac{\der\Phi_{\A,a}}{\der z_j}\Big],
\qquad j\in J.
\eean
In particular, if the weight $a$ is such that $(\A(x),a)$ is unbalanced,
then the number of solutions of system \Ref{eq for L and gamma}, counted with multiplicities,
equals $|U(\A(x))|$ and can be calculated in terms of the matroid associated with the arrangement $\A(x)$, see
 Lemma \ref{lem finite}.

\smallskip

The isomorphism $\Psi_{\A(x),a}^*$  allows us to compare objects associated with $\OCx$ and  objects associated with $\OLx$.
For example, let $q_I,p_{\bar I}$ be an ordered coordinate system on $L_{Y,a}$ like in Lemma \ref{I-coordinates lem new}.
Assume that $x \in \C^n$ is such that $\dim \OLx $ is finite. Consider the Grothendieck residue $\mc R :\OLx\to\C$,
\bean
\label{res on OLx}
[f] \ \mapsto \
\frac{1}{(2\pi i)^n} \int_{\Gamma}
\frac{f\ dp_I\wedge dq_{\bar I}}{\prod_{j=1}^n q_j }\ ,
\eean
where the differentials are ordered as in the ordered coordinate system $q_I,p_{\bar I}$, cf. \Ref{res map}.
Define the nondegenerate bilinear form $(\,,\,)_{L_{Y,a}(x)}$ on $\OLx$ by the formula
\bean
\label{res on OLx}
([f],[g])_{L_{Y,a}(x)} = \frac{(-1)^n}{d_I^2}\,\mc R\Big([f][g]\,{\prod}_{j=1}^n\frac{a_j}{[p_j^2]}\Big).
\eean

\begin{cor}
\label{res res}
Assume that the weight $a$ is generic in the sense of Lemma \ref{lem finite}, then the isomorphism $\Psi_{\A(x),a}^*$ identifies
the form $(\,,\,)_{L_{Y,a}(x)}$ on $\OLx$ and the form $(\,,\,)_{C_{\A(x),a}(x)}$ on $\OCx$.
\end{cor}

\begin{proof}  For generic $x\in\C^n$, the  corollary follows from Corollary \ref{cor hess jac}.
For all $x\in\C^n$, the corollary follows by continuity.
\end{proof}

\begin{rem}
Notice that the form $(\,,\,)_{L_{Y,a}(x)}$ is given by an $n$-dimensional integral while the form
$(\,,\,)_{C_{\A(x),a}(x)}$ is given by a $k$-dimensional integral.
\end{rem}

If $(\A(x),a)$ is unbalanced for some $x\in\cd$, then we have the  vector {\it bundle of algebras}
\bean
\label{algebra bundle L}
\sqcup_{x\in\C^n-\Delta} \OLx \to \C^n-\Delta\,.
\eean
The fiber isomorphism $\Psi_{\A(x),a}^*$ identifies this bundle of algebras with the bundle of algebras in \Ref{algebra bundle}.
The combinatorial and Gauss-Manin connections on the bundle of algebras in \Ref{algebra bundle} induce the corresponding connections
on the bundle in \Ref{algebra bundle L}.

For $x\in\cd$, we
 define the marked elements $p_{i_1,\dots,i_k}$ in $\OLx$ as the images under $\Psi_{\A(x),a}^*$ of the marked elements
$w_{i_1,\dots,i_k}$ in $\OCx$. By formula \Ref{markeD}, we have
\bean
\label{marked p}
p_{i_1,\dots,i_k} = d_{i_1,\dots,i_k}
p_{i_1}\dots p_{i_k}
\eean
with the skew-symmetry property
\bean
\label{skew marked p}
p_{i_{\sigma(1)},\dots,i_{\sigma(k)}} = (-1)^\sigma p_{i_1,\dots,i_k}, \qquad \on{for}\ \sigma \in \Si_k.
\eean
and the marked relations
\bean
\label{rel Ma p}
{\sum}_{j\in J} p_{j,i_2,\dots,i_k} = 0
\eean
labeled by independent subsets  $\{i_2,\dots,i_k\}$. By Corollary \Ref{marked span},
the marked elements $p_{i_1,\dots,i_k}$ span $\OLx$ as a vector space and relations
\Ref{skew marked p},  \Ref{rel Ma p} are the only linear relations between them.
This fact defines an integral structure on $\OLx$.

For a section $I=\sum_{j_1,\dots,j_k} I^{j_1,\dots,j_k} p_{j_1, \dots , j_k}$ of the bundle of algebras \Ref{algebra bundle L}
 and $j\in J$, we define
$\frac{\der I}{\der q_j} = \sum \frac{\der I^{j_1,\dots,j_k}}{\der q_j}p_{j_1, \dots , j_k}$.
This formula defines the combinatorial connection on \Ref{algebra bundle L}.

\begin{thm}
\label{thm Frob on p}
If a  section $I$ of the bundle of algebras \Ref{algebra bundle L} is flat with respect to the Gauss-Manin connection,
then it satisfies the differential equations
\bean
\label{dif eqn Lagr}
\kappa \frac{\der I}{\der q_j}(x) = [p_j]*_x I(x),
\qquad
j\in J,
\eean
where $[p_j]*_x$ is the operator of multiplication by $[p_j]$ in $\OLx$.
\end{thm}

\begin{rem}
Notice that solutions of these differential equations are given by the multidimensional
hypergeometric integrals as in \Ref{Ig w}.
Notice also the action of $[p_j]*_x$ on the marked elements $p_{i_{1},\dots,i_k}$ can be identified with action
on the marked elements $w_{i_{1},\dots,i_k}$ of the operator $K_j(x)$ from \Ref{K_j on w}.
\end{rem}

\subsection{Real solutions}
\label{Real solutions}

Assume that for any $j\in J$, the subspace $Y^\perp$ does not lie in the hyperplane $p_j=0$.
Assume that all coordinates of the weight $a\in(\C^\times)^n$ are positive.
Assume that all entries of the matrix $(b^i_j)_{i=1,\dots,k\atop j=1,\dots,n}$, defining $Y\subset\C^n$ in \Ref{def Y}, are real.
Assume that the critical set $C_{\A,a}$ is nonempty.

\begin{cor}
\label{RE}
Under these assumptions, if $x\in\R^n\subset\C^n$, then
all solutions of system  \Ref{eq for L and gamma}
are real and nondegenerate, and the number of solutions equals $|\chi(U(\A(x))|$.

\end{cor}

\begin{proof}
If $a\in (\R_{>0})^n$, $x\in\R^n$, and $(b^i_j)$ are real, then all points of the critical set $C_{\A(x),a}$ are real, nondegenerate, and
 the number of points equals $|\chi(U(\A(x))|$, see \cite{V3}. Now the corollary follows from Theorem \ref{thm 1}.
\end{proof}

The reality property in Corollary \ref{RE} is similar to the reality property of Schubert calculus, see \cite{MTV2, MTV3, So}.

\section{Characteristic variety of the Gauss-Manin differential equations}
\label{sec last}

Consider the Gauss-Manin differential equations $\kappa \frac{\der I}{\der z_j} = K_jI$
 in \Ref{dif eqn}.   Define the characteristic
variety of the $\kappa$-dependent $D$-module associated with the Gauss-Manin differential equations  as
\bean
\label{spec}
\phantom{aa}
\qquad
\Spect_{\A,a} = \{(x,y)\in (\cd)\times (\C^n)^*\ |\ \exists
v\in \sing_aV\ \text{with}\ K_j(x)v = y_j v,\ j\in J\}.
\eean
Let $\pi_{\Spect_{\A,a}} : \Spect_{\A,a} \to \C^n$ be the projection to $\C^n$.

Recall the Lagrangian variety $L_{Y,a}\subset \C^n\times\Cs$ introduced  in Section \ref{An arrangement second part}
and  the projection $\pi_{L_{Y,a}} : L_{Y,a}\to\C^n$.

\begin{thm}
\label{char var thm}
Assume that the weight $a$ is generic in the sense of Lemma \ref{lem finite}, then
$\Spect_{\A,a}=  \pi_{L_{Y,a}}^{-1}(\cd)$.
\end{thm}

\begin{proof}
For generic $x\in\cd$, the special vectors $(\on{F}(u))_{u\in C_{\A,a}}$ form a basis of $\Sing_aV$ by Theorem \ref{first theorem}.
This gives $\pi^{-1}_{\Spect_{\A,a}}(x) = \pi_{L_{Y,a}}^{-1}(x)$ by  Theorems \ref{thm alpha} and \ref{K/f}.
We get the equality  $\pi^{-1}_{\Spect_{\A,a}}(x) = \pi_{L_{Y,a}}^{-1}(x)$ for  all  $x\in\cd$ by continuity.
\end{proof}

Theorem \ref{char var thm} is proved in \cite{V6} if  $\A$ is generic.

\end{document}